\documentclass[11pt,reqno]{amsart}
\usepackage[T1]{fontenc}
\usepackage[utf8]{inputenc}
\usepackage{lmodern}
\usepackage[a4paper,textwidth=150mm,textheight=240mm]{geometry}
\usepackage{amsmath,amssymb,amsthm,mathtools}
\usepackage{enumitem}
\usepackage{needspace}
\usepackage{booktabs,array,mathrsfs}
\usepackage{microtype}
\usepackage[hidelinks]{hyperref}
\hypersetup{
 pdftitle={Selection Spectra and Sharp Perturbations of Normal Sequences},
 pdfauthor={Wieslaw Sliwa},
 pdfsubject={Selection spectra and perturbations of normal sequences},
 pdfkeywords={strongly normal sequences, selection spectra, Lorentz sequence spaces, symmetric sequence ideals, perturbations}
}
\setlist[enumerate]{label=\textup{(\arabic*)},leftmargin=*,itemsep=3pt,topsep=6pt}
\newtheorem{theorem}{Theorem}[section]
\newtheorem{proposition}[theorem]{Proposition}
\newtheorem{lemma}[theorem]{Lemma}
\newtheorem{corollary}[theorem]{Corollary}
\theoremstyle{remark}
\newtheorem{remark}[theorem]{Remark}

\makeatletter
\def\@begintheorem#1#2[#3]{%
  \deferred@thm@head{\the\thm@headfont \thm@indent
    \@ifempty{#1}{\let\thmname\@gobble}{\let\thmname\@iden}%
    \@ifempty{#2}{\let\thmnumber\@gobble}{\let\thmnumber\@iden}%
    \let\thmnote\@gobble
    \thm@swap\swappedhead\thmhead{#1}{#2}{#3}%
    \the\thm@headpunct
    \thmheadnl
    \hskip\thm@headsep
  }%
  \ignorespaces}
\makeatother
\numberwithin{equation}{section}
\newcommand{\N}{\mathbb N}
\newcommand{\R}{\mathbb R}
\newcommand{\cP}{\mathcal P}
\newcommand{\cS}{\mathcal S}
\newcommand{\cB}{\mathcal B}
\newcommand{\lin}{\operatorname{span}}

\title[Selection spectra and perturbations]
{Selection spectra and sharp perturbations of normal sequences}
\author{Wies\l aw \'{S}liwa}
\address{Institute of Mathematics, Faculty of Exact and Technical Sciences,
University of Rzesz\'ow, Pigonia 1, 35-310 Rzesz\'ow, Poland}
\email{wsliwa@ur.edu.pl}
\date{3 November 2026}
\subjclass[2020]{Primary 46B10, 46B15; Secondary 46B26, 46E30}
\keywords{strongly normal sequence, summability spectrum, subsequence,
block sequence, Lorentz sequence space, symmetric sequence ideal, perturbation}

\begin{document}
\begin{abstract}
For a bounded sequence $y=(y_n)\subset E^*$ and $p>0$, put
$G_p(y)=\{x\in E:(y_n(x))\in\ell_p\}$ and call $y$ $p$-strongly normal
when $G_p(y)$ is dense in $E$. We study the density spectra of the original
sequence, its subsequences, and its normalized successive blocks.
The possible subsequence and block spectra are exactly
\[
 \varnothing,\quad (0,\infty),\quad (1,\infty),\quad
 [r,\infty),\quad(r,\infty)\qquad(r>1),
\]
and their possible pairs are precisely the nested pairs. Every such pair
has a normalized weakly null $1$-unconditional basic realization in
$\ell_1(\mathfrak c)^*$, even with empty original-sequence spectrum.
Without a basicness requirement, the possible original-sequence spectra
are all upper subsets of $(0,\infty)$; each is realized in a separable
Hilbert space.

We determine the exact universal perturbation classes: $\ell_p$ for a fixed
original-sequence exponent, $c_0$ for the subsequence spectrum, and $\ell_1$
for the block spectrum. The ideal-valued version holds for every proper
solid permutation-invariant linear ideal $c_{00}\subset A\subsetneq c_0$,
without a norm or completeness assumption.

For the author's 2012 selection model, the bounded-run criterion extends
to all proper symmetric quasi-Banach sequence ideals. In the unbounded-run
case, a single norm-one projection annihilates the summability domains
of any family of fewer than $\mathfrak b$ such ideals; the quotient by their
closed linear span contains a $1$-complemented isometric copy of
$\ell_\infty$. Such copies also occur in the left jump at every $q>1$
and in the right jump at every $q\ge1$ of the scale of closed summability
domains. Thus the common density criterion coexists with two-sided
complemented discontinuities.
\end{abstract}
\maketitle

\section{Introduction}
A normalized sequence $(y_n)$ in the dual $E^*$ of a Banach space is
\emph{normal} when it is weak-star null. Following
\cite{sliwa1996,sliwa2012}, it is \emph{strongly normal} when
\[
 G_1(y)=\Bigl\{x\in E:\sum_n|y_n(x)|<\infty\Bigr\}
\]
is dense in $E$. A Banach space has an infinite-dimensional separable
quotient if and only if its dual contains a strongly normal sequence
\cite{sliwa2012}; see \cite{mujica1997,survey2018} for background. The
Josefson--Nissenzweig theorem \cite{josefson,nissenzweig} supplies a normal
sequence in the dual of every infinite-dimensional Banach space. Here we
study density of the domains
\[
 G_p(y)=\{x\in E:(y_n(x))\in\ell_p\},\qquad 0<p<\infty,
\]
and its behaviour under selection and perturbation, rather than propose
a new solution of the general separable quotient problem.

The distinction between density and all-point summability is essential.
For $p\ge1$, $G_p(y)=E$ is the classical weak-star $p$-summability
condition \cite{fourie}. Density is strictly weaker: the coordinate
functionals on $E=\ell_2$ satisfy
$G_1(y)=\ell_1\subsetneq\ell_2$, but $G_1(y)$ is dense in $\ell_2$.
We write $\cP(y)$ for the exponents at which the original sequence has
a dense domain, $\cS(y)$ for those obtained by subsequences, and
$\cB(y)$ for those obtained by normalized successive blocks of a
Schauder basic sequence. Thus $\cP$ concerns the original sequence,
$\cS$ subsequences, and $\cB$ blocks; always
$\cP\subseteq\cS\subseteq\cB$ when the last spectrum is defined.

\paragraph{Scope.}
The classification is over all Banach duals, not inside every fixed
predual. By Corollary~\ref{pre:small-density}, if
$\operatorname{dens}E<\mathfrak b$, then every normal sequence in $E^*$
has $\cS=(0,\infty)$, and every normal basic sequence has
$\cB=(0,\infty)$. Proper \emph{subsequence and block} spectra of normal
sequences are therefore large-density phenomena. This restriction does
not apply to $\cP$: Theorem~\ref{orig:classification} realizes every upper
subset of $(0,\infty)$ as an original-sequence spectrum over a separable
Hilbert predual, without requiring basicness. The complete joint
classification below concerns $(\cS,\cB)$; it does not classify all
triples $(\cP,\cS,\cB)$ for one basic sequence.

\paragraph{Contributions and prerequisites.}
The following map separates quoted inputs, auxiliary arguments reproduced
for completeness, and the results established here.
\begin{center}
\small
\renewcommand{\arraystretch}{1.08}
\begin{tabular}{@{}p{.25\textwidth}p{.70\textwidth}@{}}
\toprule
Status & Results and role\\
\midrule
Classical inputs & The selection theorem \eqref{pre:selection} is
\cite[Theorem~1]{sliwa2012}. The Lorentz block facts are
\cite[Remark~1.4 and Proposition~1.8]{lorentzideals}.\\[4pt]
Reproved auxiliary tools & Lemmas~\ref{pre:allp}--\ref{pre:nostrong}
and \ref{pre:defect} give the finite-support, reflexive-span,
canonical-bidual, and diagonal tools; their overlap with earlier work
is specified below.\\[4pt]
Spectral results & Theorems~\ref{lor:weight}, \ref{lor:classification}
and \ref{joint:classification} give the hereditary weight criterion and
the individual and joint selection classifications. Theorems
\ref{lor:open-space} and \ref{orig:classification} give spaceable open
thresholds and the separate classification of $\cP$ without basicness.\\[4pt]
Sharp perturbations & Theorem~\ref{pert:thresholds},
Corollary~\ref{pert:wholeP}, and Theorem~\ref{ideal:perturb} determine the
exact universal error classes, with explicit counterexamples outside them.\\[4pt]
Extensions of the 2012 model & Section~\ref{sec:jmsj} extends the
exponent-one run-length criterion to quasi-Banach ideals.
Theorems~\ref{jmsj:common-projection} and \ref{jmsj:strict-scale} give
common complemented obstructions and two-sided jumps of closed domains.\\[4pt]
Common realization & Corollary~\ref{fixed:joint} places every admissible
pair in the single dual $\ell_1(\mathfrak c)^*$, with a normalized weakly
null $1$-unconditional realization and, if desired, $\cP=\varnothing$.\\
\bottomrule
\end{tabular}
\end{center}

\paragraph{Relation to earlier work.}
The finite-support argument of Lemma~\ref{pre:allp} already appears in
\cite[proof of Proposition~2]{sliwa2012}; we record its simultaneous-exponent
and quotient consequences. For the companion preprint
\cite{sliwa2026}, references here are to arXiv:2610.12787v1.
The coefficient-vector argument of Lemma~\ref{pre:allp} also occurs in
its proof of Theorem~9, $(2)\Rightarrow(3)$. Lemma~\ref{pre:reflexive}
uses the weak-star closedness and quotient identification of its
Proposition~10, Steps~1--2, and records the basis consequence.
Lemma~\ref{pre:defect} uses the diagonal method of its Theorem~9,
$(3)\Rightarrow(1)$, now for a quotient by a closed union of summability
domains. By contrast, Lemma~\ref{pre:nostrong} uses separability of
$[z_n]^*$ in a canonical bidual; it is not a restatement of the
Grothendieck-specific obstruction in its Theorem~26.
Lemmas~\ref{pre:allp}, \ref{pre:reflexive}, \ref{pre:nostrong}, and
\ref{pre:defect} are auxiliary tools reproduced for completeness; no
novelty is claimed for these lemmas. Their relation to earlier arguments
is specified above, and all four are proved below. Neither the companion
paper's reflexivity characterization nor its $L_1$ realizations are
used as logical inputs or claimed anew here.

The negative existence result of \cite{kakol2026,sliwa2026} is also not
claimed as new: the empty-spectrum case of Corollary~\ref{fixed:joint}
recovers the relevant phenomenon in $\ell_1(\mathfrak c)^*$. The
additional content is the full finite-exponent and block picture, all
nested pairs, and the sharp perturbation and closed-domain results.
Likewise, Section~\ref{sec:jmsj} starts from the model and exponent-one
criterion of \cite{sliwa2012}; its simultaneous quasi-Banach-ideal and
complemented-geometry conclusions are the extensions proved here.

\paragraph{Proof route.}
The exponent-one selection theorem and Lemma~\ref{pre:allp} force either
selection spectrum to collapse to $(0,\infty)$ if it meets $(0,1]$.
Upper-set monotonicity then gives the list of possible shapes.
For realization, Theorem~\ref{lor:weight} computes the allowed exponents
from a Lorentz weight and excludes the forbidden union for every
normalized sequence in the relevant span. This proves
Theorem~\ref{lor:classification}. Theorem~\ref{pert:thresholds} supplies
the sharp error classes: $\ell_p$ for a fixed original exponent,
$c_0$ for $\cS$, and $\ell_1$ for $\cB$. Two Lorentz models and an
$\ell_1$-saturated mixed construction then yield
Theorem~\ref{joint:classification}; metric-quotient transfer yields
Corollary~\ref{fixed:joint}.
The spaceability, unrestricted $\cP$, and run-length constructions are
separate branches, not prerequisites for that pair classification.
Theorem~\ref{ideal:perturb} generalizes the fixed-exponent perturbation
argument; the earlier explicit Hilbert-space proof is retained as a
concrete paired-coordinate model for the three-part sharpness theorem.

\paragraph{Related context.}
Exponent sets also occur in the Banach--Saks theory \cite{banachsaks},
but there the quantifiers concern sequences throughout a space; our
invariants belong to one fixed sequence of functionals. Work on
$p$-summability and weak maximizing properties \cite{blanco}, and on
symmetric sequence ideals \cite{defant,pietsch}, provides related but
different frameworks. For adjacent Josefson--Nissenzweig questions in
function spaces, see \cite{banakhjnp} for $C_p$-spaces and
\cite{kakolsobotazdomskyy} for Grothendieck-type properties of $C(K)$.
These are contextual references, not sources of the spectral
classification proved here.

Weak-star nullity is part of normality; for a bounded sequence it also
follows from density of a nonempty union of summability domains
(Lemma~\ref{pre:elementary}(ii)). Basicness is required for $\cB$, not
for the definition of $\cP$ or $\cS$. Throughout, a dense $G_p$-domain
must be distinguished from the stronger identity $G_p=E$.
Section~\ref{sec:questions} records the remaining classification and
cardinal questions and a short interleaving example separating
$\cP$ from $\cS$ for a basic sequence.

\section{Summability domains and selection principles}\label{sec:prelim}

All spaces and scalars are real. We write \(B_E\) and \(S(E)\) for the
closed unit ball and the unit sphere, respectively, \([x_n]\) for the
norm-closed linear span of a sequence, and
\(J_E:E\to E^{**}\) for the canonical embedding. Structural
subspaces are norm closed; algebraic spans and summability domains
need not be. A Banach space is \emph{hereditarily $\ell_1$}, or equivalently
\emph{$\ell_1$-saturated}, if each of its infinite-dimensional subspaces
contains an isomorphic copy of $\ell_1$.
We write $f\le^*g$ for eventual domination in $\N^\N$ and denote by
$\mathfrak b$ the least cardinality of a family in $\N^\N$ which is
unbounded for $\le^*$.

For a bounded sequence \(y=(y_n)\subset E^*\) and \(0<q<\infty\), put
\[
 G_q(y)=\{x\in E:(y_n(x))\in\ell_q\}.
\]
We call \(y\) \emph{\(q\)-strongly normal} if \(G_q(y)\) is dense in \(E\).
A \emph{normal} sequence is a normalized weak-star null sequence.
We shall also work with seminormalized sequences, for which
\(0<\inf_n\|y_n\|\le\sup_n\|y_n\|<\infty\).
For \(q<1\), only membership in the usual quasi-Banach sequence space
\(\ell_q\) is involved in the definition.

The \emph{summability spectrum} and the \emph{subsequence spectrum} are
\[
\begin{split}
 \cP(y)&=\{q>0:G_q(y)\text{ is dense in }E\},\\
 \cS(y)&=\{q>0:\text{some subsequence of \(y\) is \(q\)-strongly normal}\}.
\end{split}
\]
For a normalized Schauder basic sequence \(y\), we also define
\[
 \cB(y)=\{q>0:\text{\(y\) has a normalized \(q\)-strongly normal
 block sequence}\}.
\]
Blocks are finite linear combinations over successive disjoint finite
intervals; zero coefficients and singleton blocks are allowed.
We also use the phrase \emph{successive blocks} for such finite
combinations of an arbitrary sequence, without asserting basicness.
The block spectrum \(\cB\) is defined only for Schauder basic
sequences.  Thus all three symbols record density properties; the stronger
all-point condition is always written explicitly as $G_q(y)=E$.
For a nonempty set \(A\subset(0,\infty)\), set
\[
 D_A(y)=\bigcup_{q\in A}G_q(y).
\]
In particular,
\[
 D_{<p}(y)=\bigcup_{0<q<p}G_q(y),\qquad
 D_{<\infty}(y)=\bigcup_{q>0}G_q(y).
\]
Each of these sets is a linear subspace: the domains \(G_q(y)\) are
linear and increase with \(q\). The assertion that \(D_A(y)\) is
nondense is stronger than the separate assertions that \(G_q(y)\)
is nondense for every \(q\in A\).

\begin{lemma}\label{pre:elementary}
Let \(y\subset E^*\) be bounded.
\begin{enumerate}[label=\textup{(\roman*)}]
\item If \(0<q<t<\infty\), then \(G_q(y)\subset G_t(y)\).
Consequently, each of \(\cP(y)\), \(\cS(y)\), and, when defined,
\(\cB(y)\) is an upper subset of \((0,\infty)\).
\item If \(D_A(y)\) is dense for some nonempty \(A\subset(0,\infty)\),
then \(y_n\to0\) weak-star.
\item If \(0<c\le|\lambda_n|\le C<\infty\), then
\(G_q((\lambda_ny_n))=G_q(y)\) for every fixed \(q>0\).
\item For every normalized normal sequence, \(\cP(y)\subseteq\cS(y)\); when \(y\) is Schauder basic, \(\cS(y)\subseteq\cB(y)\).
\end{enumerate}
\end{lemma}
\begin{proof}
The first assertion follows from \(\ell_q\subset\ell_t\). Every scalar
sequence obtained by evaluating \(y\) on \(D_A(y)\) tends to zero.
If \(M=\sup_n\|y_n\|\), approximation of \(x\in E\) by
\(u\in D_A(y)\) gives
\[
 \limsup_n|y_n(x)|\le M\|x-u\|,
\]
which proves (ii). Finally, multiplication by \((\lambda_n)\)
preserves membership in each fixed \(\ell_q\).  For (iv), the original
sequence is an allowed subsequence and every subsequence is a block sequence
with singleton supports.
\end{proof}

When \(q\ge1\), the stronger condition \(G_q(y)=E\) is the usual
weak-star \(q\)-summability condition; see \cite{fourie}.
The closed graph theorem then gives a bounded evaluation operator
\begin{equation}
S_y:E\longrightarrow\ell_q,\qquad S_yx=(y_n(x)).
 \label{pre:operator}
\end{equation}
Indeed, convergence in \(E\) and in \(\ell_q\) identifies every
coordinate of a possible graph limit. For \(q>1\), duality yields
\begin{equation}
\left\|\sum_n a_ny_n\right\|
 \le \|S_y\|\,\|(a_n)\|_{q'}\qquad(a\in c_{00}),
 \quad q'=\frac q{q-1}.
 \label{pre:upper}
\end{equation}
Conversely, such an estimate implies \(G_q(y)=E\).
There need be no bounded operator \(E\to\ell_q\) when only
\(\overline{G_q(y)}=E\) is assumed.

\subsection{Exponent one: cited background}
Following \cite{sliwa2012}, a sequence in $E^*$ is called
\emph{weak-star Schauder basic} if it is a Schauder basis of its
$\sigma(E^*,E)$-closed linear span.  In particular, its coefficient
functionals are continuous for the induced weak-star topology.  The
classical framework of weak-star basic sequences goes back to
Johnson--Rosenthal \cite{johnsonrosenthal}.  The one external
exponent-one selection result used below is the following.

\begin{equation}\label{pre:selection}
\begin{aligned}
&\text{Every normalized strongly normal sequence in $E^*$ has a}\\
&\text{weak-star Schauder basic subsequence.}
\end{aligned}
\end{equation}
This is exactly \cite[Theorem~1]{sliwa2012}.

The coefficient-vector argument in the proof of
\cite[Proposition~2]{sliwa2012} already gives the finite-support dense
subspace used in the next lemma. We record its simultaneous-exponent
formulation and the detected-quotient conclusion for use below; the
underlying argument is not new.

\begin{lemma}[Finite-support evaluation domain]\label{pre:allp}
Let $z=(z_n)\subset E^*$ be a seminormalized weak-star Schauder basic
sequence.  There is a dense linear subspace $D\subset E$ such that
$(z_n(x))$ has finite support for every $x\in D$.  Consequently $z$ is
$q$-strongly normal for every $q>0$, and
$E/\bigcap_n\ker z_n$ is separable.
\end{lemma}
\begin{proof}
Let
\[
 Y=\overline{\operatorname{span}\{z_n:n\in\N\}}^{\,\sigma(E^*,E)}.
\]
For a weak-star Schauder basic sequence, the coefficient functionals
$\varphi_n$ on $Y$ are weak-star continuous.  Hence for every $n$ there
is $x_n\in E$ such that
\[
 \varphi_n(y)=y(x_n)\qquad(y\in Y).
\]
Put $K=\bigcap_n\ker z_n$.  Then $K={}^\perp Y$.  We claim that
\[
 \overline{K+\operatorname{span}\{x_n:n\in\N\}}=E.
\]
Indeed, if the image of $\operatorname{span}\{x_n\}$ were not dense
in $E/K$, Hahn--Banach would give a nonzero element $y\in K^\perp=Y$
vanishing on every $x_n$.  Then all coefficients
$\varphi_n(y)=y(x_n)$ would vanish, contradicting uniqueness of the
Schauder expansion of $y$.

Set
$D=K+\operatorname{span}\{x_n:n\in\N\}$.  The claim shows that
$D$ is dense.  Moreover, for
$x=k+\sum_{j=1}^m a_jx_j\in D$ one has
$z_n(k)=0$ and $z_n(x_j)=\varphi_j(z_n)=\delta_{jn}$, so
$(z_n(x))$ has finite support.  Thus $D\subset G_q(z)$ for every
$q>0$, proving simultaneous $q$-strong normality.  Finally, the cosets
of the $x_n$ have dense linear span in $E/K$, so $E/K$ is separable.
\end{proof}

\begin{lemma}[Reflexive spans]\label{pre:reflexive}
Every norm-closed reflexive subspace $Y$ of a dual $E^*$ is weak-star
closed, and every Schauder basis of $Y$ is weak-star Schauder basic in
$E^*$.
\end{lemma}
\begin{proof}
Let $Y\subset E^*$ be norm closed and reflexive.  Its unit ball $B_Y$ is
weakly compact, and the inclusion $Y\hookrightarrow E^*$ is
weak-to-weak-star continuous.  Hence $B_Y$ is weak-star compact, and
therefore weak-star closed in $E^*$.  By the Krein--\v{S}mulian theorem,
a convex subset of a dual space is weak-star closed once its intersection
with every closed ball is weak-star closed.  Applied to the linear subspace
$Y$, this shows that $Y$ is weak-star closed.

Let $Y_\perp=\{x\in E:y(x)=0\text{ for every }y\in Y\}$ and put
$X=E/Y_\perp$.  Since $Y$ is weak-star closed, $X^*$ identifies
isometrically with $Y$.  The dual $X^*=Y$ is reflexive, hence $X$ is
reflexive and the canonical map $X\to Y^*$ is onto.  Therefore every
coefficient functional of every Schauder basis of $Y$ is evaluation at
some element of $E$ modulo $Y_\perp$.  Norm convergence of the basis
expansion implies weak-star convergence, so the basis is a Schauder basis
of the weak-star closed space $Y$ and is weak-star Schauder basic in $E^*$.
\end{proof}

\begin{lemma}[Hereditary $\ell_1$ obstruction]\label{pre:nostrong}
Let $Z$ be hereditarily $\ell_1$, put $E=Z^*$, and let
$Y=J_ZZ\subset Z^{**}=E^*$.  Then $Y$ contains no normalized strongly
normal sequence, and every weak-star null sequence in $Y$ is weakly null
as a sequence in $Y$.
\end{lemma}
\begin{proof}
If a normalized strongly normal sequence existed in $J_ZZ$, pass to the
weak-star Schauder basic subsequence given by \eqref{pre:selection} and
write it as $(J_Zz_n)$.  By Lemma~\ref{pre:allp}, the quotient
$Z^*/\bigcap_n\ker(J_Zz_n)$ is separable. The restriction map
$Z^*\to[z_n]^*$ is onto by Hahn--Banach and has kernel
$[z_n]^\perp=\bigcap_n\ker(J_Zz_n)$. Hence
\[
 Z^*/[z_n]^\perp\cong [z_n]^*
\]
is separable.  But the infinite-dimensional subspace $[z_n]$ of the
hereditarily $\ell_1$ space $Z$ contains a copy of $\ell_1$, so
$\ell_\infty$ is a quotient of $[z_n]^*$, a contradiction.  Finally,
weak-star convergence of $J_Zz_n$ in $Z^{**}$ is precisely weak
convergence of $z_n$ in $Z$.
\end{proof}

\begin{proposition}[Amplification by successive blocks]
\label{pre:amplification}
Let \(p>1\), and let \(y=(y_n)\subset E^*\) be bounded. If
\(G_p(y)\ne E\), there are successive finite intervals \(I_k\)
and scalars \(a_n^{(k)}\) such that
\[
 u_k=\sum_{n\in I_k}a_n^{(k)}y_n,\qquad
 \|u_k\|=1,\qquad G_p(y)\subset G_1(u).
\]
If \(G_p(y)\) is proper and dense, then \(E\) has an
infinite-dimensional separable quotient, and one may choose a
further subsequence of these blocks that is weak-star Schauder
basic and \(q\)-strongly normal simultaneously for every \(q>0\).
In particular, if \(y\) is normalized and Schauder basic in this
case, then \(\cB(y)=(0,\infty)\).
Moreover, if a linear subspace \(Y\subset E^*\) contains no
normalized strongly normal sequence, then every bounded
\(y\subset Y\) satisfies
\[
 \overline{G_p(y)}=E\quad\Longleftrightarrow\quad G_p(y)=E.
\]
\end{proposition}
\begin{proof}
Put \(p'=p/(p-1)\). No tail of \(y\) satisfies an upper
\(\ell_{p'}\) estimate. Indeed, adding the finitely many omitted
coordinates would yield such an estimate for the whole sequence,
and the converse to \eqref{pre:upper} would give \(G_p(y)=E\).
Starting with \(N_0=0\), choose \(N_k>N_{k-1}\) and a scalar
family \(b^{(k)}\) supported in \(I_k=(N_{k-1},N_k]\) such that
\[
 d_k:=\left\|\sum_{n\in I_k}b_n^{(k)}y_n\right\|
       >2^k\|b^{(k)}\|_{p'}.
\]
Set \(a^{(k)}=b^{(k)}/d_k\). H\"older's inequality gives
\[
 |u_k(x)|\le 2^{-k}\|(y_n(x))\|_p
 \qquad(x\in G_p(y)).
\]
Summation proves the inclusion. If \(G_p(y)\) is dense, then
\(u\) is normalized and strongly normal. Apply
\eqref{pre:selection} and Lemma~\ref{pre:allp}; a subsequence
of successive blocks is still a successive block sequence.
The detected quotient is infinite dimensional because the
selected sequence is basic. Finally, all \(u_k\) belong to
\(Y\) when \(y\subset Y\), proving the last assertion.
\end{proof}

\begin{remark}\label{pre:operator-range-credit}
The separable-quotient consequence also follows from the classical
Saxon--Wilansky characterization by proper dense operator ranges
\cite{saxonwilansky}. Indeed, \(G_p(y)\), endowed with the graph
norm \(\|x\|+\|(y_n(x))\|_p\), is Banach and its inclusion
into \(E\) is a bounded operator with proper dense range.
Proposition~\ref{pre:amplification} additionally produces successive
normalized blocks of the given sequence.
\end{remark}

\begin{corollary}[Vector-dependent exponents]
\label{pre:union-amplification}
Let \(\varnothing\ne A\subset(1,\infty)\), and let
\(y\subset E^*\) be bounded. If \(G_q(y)\ne E\) for every
\(q\in A\), there is a normalized successive block sequence
\(u\) of \(y\) such that
\[
 D_A(y)\subset G_1(u).
\]
Consequently, if \(D_A(y)\) is dense, then \(E\) has an
infinite-dimensional separable quotient; when \(y\) is
normalized and Schauder basic, \(\cB(y)=(0,\infty)\).
\end{corollary}
\begin{proof}
Choose a nondecreasing cofinal sequence \((q_k)\) in \(A\),
constant if \(A\) has a largest element. At stage \(k\),
apply the tail argument in Proposition~\ref{pre:amplification}
with exponent \(q_k\) and amplification factor \(2^k\).
For \(x\in G_q(y)\), \(q\in A\), eventually \(q_k\ge q\),
so H\"older's inequality and the inclusion of scalar sequence
spaces give
\[
 |u_k(x)|\le2^{-k}\|(y_n(x))\|_{q_k}
            \le2^{-k}\|(y_n(x))\|_q.
\]
The finitely many remaining terms are harmless. The conclusions
follow as in Proposition~\ref{pre:amplification}.
\end{proof}

\begin{corollary}[A dense union in a hereditary obstruction space]
\label{pre:union-allpoint}
Suppose that a linear subspace \(Y\subset E^*\) contains no
normalized strongly normal sequence. For every seminormalized
sequence \(y\subset Y\) and every nonempty \(A\subset(0,\infty)\),
\[
 \overline{D_A(y)}=E
 \quad\Longleftrightarrow\quad
 G_q(y)=E\text{ for some }q\in A.
\]
\end{corollary}
\begin{proof}
Only the forward implication requires proof. If \(A\subset(0,1]\),
then \(D_A(y)\subset G_1(y)\), so bounded normalization of \(y\)
would give a strongly normal sequence in \(Y\).
Otherwise \(A\cap(1,\infty)\) is cofinal in \(A\).  Moreover
\(D_A(y)=D_{A\cap(1,\infty)}(y)\), since any exponent of \(A\) at most
one is dominated by every exponent of \(A\cap(1,\infty)\).
Corollary~\ref{pre:union-amplification} therefore gives a strongly normal
block sequence in \(Y\) unless \(G_q(y)=E\) for some \(q\in A\).
\end{proof}

\begin{corollary}\label{pre:block-density-alternative}
Let \(y\) be a normalized Schauder basic sequence such that
\(\cB(y)\ne(0,\infty)\). For every normalized block sequence
\(v\) of \(y\) and every \(p>1\),
\[
 \overline{G_p(v)}=E\quad\Longleftrightarrow\quad G_p(v)=E.
\]
\end{corollary}
\begin{proof}
Otherwise Proposition~\ref{pre:amplification} would give one
block sequence of \(v\), and hence of \(y\), which is
\(q\)-strongly normal for all \(q>0\).
\end{proof}

\subsection{Perturbations and the size of the density defect}

\begin{lemma}\label{pre:perturb}
Suppose \(u,v\subset E^*\) and
\((\|u_n-v_n\|)\in\ell_q\) for some \(q>0\).
Then \(G_q(u)=G_q(v)\). If the errors belong to every \(\ell_q\),
then \(D_A(u)=D_A(v)\) for every nonempty \(A\subset(0,\infty)\).
\end{lemma}
\begin{proof}
For each \(x\in E\),
\(|(u_n-v_n)(x)|\le\|u_n-v_n\|\|x\|\).
The evaluation difference therefore belongs to \(\ell_q\);
linearity of \(\ell_q\) proves the assertion, also when \(q<1\).
\end{proof}

\begin{lemma}[Silent direct-sum factors]\label{pre:directsum}
Let $E=E_1\oplus_1 E_2$, so $E^*=E_1^*\oplus_\infty E_2^*$, and let
$u_n=(v_n,0)$ with $v_n\in E_1^*$.  Then for every $q>0$,
\[
 G_q^E(u)=G_q^{E_1}(v)\oplus E_2.
\]
Consequently density and all-point $q$-summability for $u$ in $E$ are
equivalent to the corresponding properties for $v$ in $E_1$; the same
is true after passage to subsequences.  Moreover, weak-star nullity of a
sequence in $E_1^*\oplus\{0\}$ is exactly weak-star nullity of its
first coordinate, and weak nullity in the first-coordinate subspace is
exactly weak nullity in $E_1^*$.
\end{lemma}
\begin{proof}
For $(x_1,x_2)\in E$ one has $u_n(x_1,x_2)=v_n(x_1)$, which gives the
domain identity and all density assertions.  The convergence statements
follow directly from the product dualities.
\end{proof}

We will repeatedly use the standard gliding-hump selection principle,
a basic consequence of the Bessaga--Pe\l czy\'nski selection method
\cite{bessagapelczynski}, in the following elementary form. If a Banach space has a Schauder basis
and \(v_n\) is a normalized weakly null sequence, then, after passing
to a subsequence, there are normalized successive blocks \(b_k\)
such that
\begin{equation}
\|v_{n_k}-b_k\|\le 2^{-k^2}.
 \label{pre:gliding}
\end{equation}
To obtain this estimate, make the finitely many previously used
coordinates small by weak nullity, cut off the remaining tail in
norm, and normalize the resulting finite block. At each step the
two truncation errors can be chosen smaller than any prescribed
positive number. Thus Lemma~\ref{pre:perturb} applies simultaneously
to all finite exponents.

\begin{corollary}[Subsequence properties under norm-null errors]
\label{pre:norm-null-subsequences}
Let \(u,v\subset E^*\) be bounded and suppose that
\(\|u_n-v_n\|\to0\). Then \(u\) has a subsequence which is
\(q\)-strongly normal simultaneously for all \(q>0\) if and
only if \(v\) has such a subsequence. More generally, for every
fixed nonempty \(A\subset(0,\infty)\), the existence of a
subsequence \(z\) with dense \(D_A(z)\) is the same for \(u\)
and \(v\).
\end{corollary}
\begin{proof}
Thin a subsequence having the stated property until its errors
are bounded by \(2^{-k^2}\). Passing to a subsequence only
enlarges every summability domain. Lemma~\ref{pre:perturb}
then identifies the two selected \(G_q\)-domains for every
\(q>0\). It also identifies their unions over \(A\).
Interchange \(u\) and \(v\) for the converse.
\end{proof}

\begin{lemma}[Separable-defect recovery]\label{pre:defect}
Let \(y\subset E^*\) be weak-star null, and let
\(\varnothing\ne A\subset(0,\infty)\). If
\[
 E/\overline{D_A(y)}
\]
is separable, then \(y\) has a subsequence \(z\) for which
\(D_A(z)\) is dense in \(E\).
\end{lemma}
\begin{proof}
Put $M=\overline{D_A(y)}$.  Choose representatives $x_j\in E$ such
that the linear span of the cosets $x_j+M$ is dense in $E/M$.  Since
$y_n\to0$ weak-star, choose inductively
$n_1<n_2<\cdots$ so that
\[
 |y_{n_k}(x_j)|\le 2^{-k^2}\qquad(1\le j\le k).
\]
For each fixed $j$ the scalar sequence $(y_{n_k}(x_j))_k$ therefore
belongs to $\ell_t$ for every $t>0$; hence
$x_j\in G_t((y_{n_k}))$ for every $t>0$.  Passing to a subsequence
preserves membership in every $G_q$, so
\[
 D_A(y)+\lin\{x_j:j\in\N\}\subset D_A((y_{n_k})).
\]
The first summand is dense in $M$, while the images of the $x_j$ have
dense linear span in $E/M$.  Thus the left-hand side is dense in $E$,
and so is $D_A((y_{n_k}))$.
\end{proof}

\begin{corollary}[Small-density collapse]\label{pre:small-density}
Let $y$ be a normalized normal sequence in $E^*$.
If $\operatorname{dens}E<\mathfrak b$, then
\[
 \cS(y)=(0,\infty).
\]
If, in addition, $y$ is Schauder basic, then
$\cB(y)=(0,\infty)$.
In fact, one subsequence is $q$-strongly normal simultaneously for every
$q>0$, with a common dense linear subspace of summable evaluations.
In particular this holds whenever $E$ is separable.
\end{corollary}
\begin{proof}
Choose a dense set $D\subset E$ of cardinality
$\kappa=\operatorname{dens}E<\mathfrak b$.  For $x\in D$ choose
$f_x\in\N^\N$ so that
\[
 |y_n(x)|\le 2^{-m}\qquad(n\ge f_x(m)).
\]
Since $|D|<\mathfrak b$, the family $\{f_x:x\in D\}$ is bounded for
eventual domination.  Choose $g\in\N^\N$ with $f_x\le^*g$ for every
$x\in D$, and then strictly increasing indices $n_m\ge g(m)$.
For each fixed $x\in D$ there is $m_x$ such that
$|y_{n_m}(x)|\le2^{-m}$ for all $m\ge m_x$. For every $q>0$,
\[
 \sum_{m=1}^\infty |y_{n_m}(x)|^q
 \le \sum_{m<m_x}|y_{n_m}(x)|^q
       +\sum_{m\ge m_x}2^{-mq}<\infty.
\]
Thus
$\operatorname{span}D\subset\bigcap_{q>0}G_q((y_{n_m}))$.
This already proves simultaneous $q$-strong normality of the same
subsequence for every $q>0$; no further use of the exponent-one selection
theorem is necessary. If $y$ is Schauder basic, the chosen subsequence
is a block sequence with singleton supports, which also proves the
assertion about $\cB$.
This is the classical small-density selection argument: see
Saxon--S\'anchez Ruiz \cite{saxonruiz1996} for the separable-quotient theorem
and \cite[Corollary~11 and its proof]{survey2018} for the strongly normal
sequence proof; compare also \cite[Proposition~3.1]{kakol2026}.
\end{proof}

\begin{corollary}\label{pre:nonsep}
Let \(\varnothing\ne A\subset(0,\infty)\).
Suppose \(Y\subset E^*\) contains no normalized sequence \(z\)
for which \(D_A(z)\) is dense. Then every normalized weak-star
null sequence \(y\subset Y\) satisfies
\[
 E/\overline{D_A(y)}\text{ is nonseparable}.
\]
\end{corollary}
\begin{proof}
A separable quotient would yield a forbidden subsequence by
Lemma~\ref{pre:defect}.
\end{proof}

Weak-star nullity is necessary here. If \(y_n=f\in S(E^*)\)
for all \(n\), then \(D_A(y)=\ker f\), whose quotient has
dimension one.

\section{Lorentz models and selection spectra}\label{sec:lorentz}

The finite-exponent obstruction can be described exactly in terms of a
scalar weight. Besides producing both open and closed critical
thresholds, this description determines all possible spectra for
selection by subsequences or by normalized blocks. We formulate the
weight theorem for embeddings into dual spaces in a form that isolates the
exact structural hypotheses used by the obstruction argument.

The proof architecture is as follows. The absence of strongly normal
sequences, together with Proposition~\ref{pre:amplification}, turns density
into all-point summability. Gliding-hump selection then reduces a hypothetical
forbidden exponent to the classical Lorentz block domination estimate.
The scalar weight decides both allowed and forbidden exponents. The later
classifications combine this criterion with the exponent-one selection
principle; the original-sequence construction at the end of this section
is independent of the Lorentz argument.

\subsection{The weight criterion}

Let \(w=(w_n)\) satisfy
\begin{equation}\label{lor:weight-assumptions}
 1=w_1\ge w_2\ge\cdots>0,\qquad
 w_n\longrightarrow0,\qquad
 \sum_{n=1}^{\infty}w_n=\infty.
\end{equation}
The Lorentz sequence space \(Z=d(w,1)\) consists of the scalar sequences
\(a\in c_0\) for which
\begin{equation}\label{lor:norm}
 \|a\|_Z=\sum_{n=1}^{\infty}a_n^*w_n<\infty,
\end{equation}
where \(a^*\) is the nonincreasing rearrangement of \((|a_n|)\).
This is the standard Banach Lorentz sequence space \(d(w,1)\), not a
weak-Lorentz space; see \cite[Section~4.e]{lindenstrausstzafriri} for
classical background. Its unit vectors \((e_n)\) form a normalized
\(1\)-symmetric Schauder basis. We shall use the following two classical
properties:
\begin{enumerate}
\item Every infinite-dimensional closed subspace of \(Z\) contains
      an isomorphic copy of \(\ell_1\).
\item If a seminormalized block sequence \((b_n)\) of \((e_n)\) has no
      subsequence equivalent to the unit vector basis of \(\ell_1\),
      then it dominates \((e_n)\). Thus there is \(C>0\) such that
      \begin{equation}\label{lor:classical-domination}
       \left\|\sum_n a_ne_n\right\|_Z
       \le C\left\|\sum_n a_nb_n\right\|_Z
       \qquad(a\in c_{00}).
      \end{equation}
\end{enumerate}
Property (1) is recorded in \cite[Remark~1.4]{lorentzideals}.
Property (2) is exactly the \emph{second assertion} of
\cite[Proposition~1.8]{lorentzideals}, with $p=1$; see also
\cite{altshuler}. It concerns the entire block sequence, under precisely
\eqref{lor:weight-assumptions}, with no additional regularity of $w$.
The word \emph{dominates} means \eqref{lor:classical-domination}.

Here is a short derivation from the small-coefficient block result in
\cite[Remark~1.4]{lorentzideals}. If
$\inf_n\|b_n\|_\infty=0$, choose a subsequence with supremum norms
tending to zero. Its concatenated coefficients tend to zero, so that
classical result supplies a further subsequence equivalent to the
$\ell_1$ basis, a contradiction. Hence
$\delta:=\inf_n\|b_n\|_\infty>0$. In each finite block choose a coordinate
$k_n$ with $|b_n(k_n)|\ge\delta$. Suppression, solidity, and symmetry give
\[
 \left\|\sum_n a_nb_n\right\|_Z
 \ge \left\|\sum_n a_nb_n(k_n)e_{k_n}\right\|_Z
 \ge \delta\left\|\sum_n a_ne_{k_n}\right\|_Z
 =\delta\left\|\sum_n a_ne_n\right\|_Z.
\]
Thus $C=\delta^{-1}$ works for the whole sequence in
\eqref{lor:classical-domination}; no further selection is needed there.
The only external step in this derivation is the cited classical
small-coefficient result.

We also record a direct verification of weak nullity. Put
\[
 W(m)=\sum_{i=1}^{m}w_i.
\]
Since \(w_n\to0\), one has \(W(m)/m\to0\).
For distinct indices \(n_1,\ldots,n_m\) and signs
\(\varepsilon_i\in\{-1,1\}\), symmetry gives
\[
 \left\|\sum_{i=1}^{m}\varepsilon_i e_{n_i}\right\|_Z=W(m).
\]
If some \(f\in Z^*\) satisfied \(|f(e_{n_i})|\ge\eta>0\) on
an infinite set, choosing the signs to agree with those of
\(f(e_{n_i})\) would give
\[
 m\eta\le \|f\|W(m)
\]
for every \(m\), a contradiction. Thus \(e_n\to0\) weakly in \(Z\).

For the rest of this subsection, write
\begin{equation}\label{lor:allowed-forbidden}
 A_w=\{q>0:w\in\ell_q\},\qquad
 F_w=(0,\infty)\setminus A_w.
\end{equation}
The assumptions on \(w\) imply \(A_w\subset(1,\infty)\), so
\((0,1]\subset F_w\).
The set \(A_w\) is an upper interval and \(F_w\) is downward closed.
For a bounded sequence \(z\subset E^*\), the corresponding forbidden
summability domain is
\[
 D_{F_w}(z)=\bigcup_{q\in F_w}G_q(z).
\]
It is a linear subspace because the sets \(G_q(z)\) are nested.

\begin{theorem}[The Lorentz weight criterion]\label{lor:weight}
Let \(Z=d(w,1)\), where \(w\) satisfies
\eqref{lor:weight-assumptions}, and let
\(J:Z\longrightarrow E^*\) be an isomorphic embedding.
Put $Y=JZ$ and assume:
\begin{enumerate}[label=\textup{(H\arabic*)}]
\item $Y$ contains no normalized strongly normal sequence;
\item every weak-star null sequence in $Y$ is weakly null as a sequence
      in $Y$.
\end{enumerate}
Then, with
\[
 y_n=\frac{Je_n}{\|Je_n\|},
\]
the sequence \((y_n)\) is a normalized weakly null Schauder basic sequence and
\begin{equation}\label{lor:weight-spectrum}
 \cP(y)=A_w.
\end{equation}
More precisely, the following assertions hold.
\begin{enumerate}
\item For every \(q\in A_w\), one has \(G_q(y)=E\).
\item For every normalized sequence \(z\subset Y\),
      \begin{equation}\label{lor:forbidden-union}
       \overline{D_{F_w}(z)}\ne E.
      \end{equation}
      In particular, no normalized sequence in \(Y\) is
      \(q\)-strongly normal for any \(q\in F_w\).
\item Every subsequence of \((y_n)\) has exactly the spectrum \(A_w\).
\end{enumerate}
Hypotheses \textup{(H1)}--\textup{(H2)} hold in particular for the
canonical embedding $J_Z:Z\to Z^{**}$ when $Z$ is hereditarily $\ell_1$,
by Lemma~\ref{pre:nostrong}.

In the canonical case, \((y_n)\) is \(1\)-symmetric.
More generally, the same conclusion holds whenever
the norm \(a\mapsto\|Ja\|\) on \(Z\) is invariant under changes of signs
and permutations of coordinates.
\end{theorem}

\begin{proof}
Fix constants \(0<c\le C_J<\infty\) such that
\begin{equation}\label{lor:embedding-bounds}
 c\|v\|_Z\le\|Jv\|\le C_J\|v\|_Z
 \qquad(v\in Z).
\end{equation}
The sequence \((Je_n)\) is Schauder basic, since \(J\)
is an isomorphism onto its closed range. Its norms lie in
\([c,C_J]\), so normalization preserves basicness.
Since \(e_n\to0\) weakly and bounded linear maps preserve weak
convergence, \(Je_n\to0\) weakly in \(E^*\).
Multiplication by the bounded scalars \(1/\|Je_n\|\) preserves weak
nullity. The symmetry assertions follow directly from the stated
invariance of the norm.

Let \(q\in A_w\). Then \(q>1\).
For finitely supported \(a\), H\"older's inequality gives
\[
 \left\|\sum_n a_ne_n\right\|_Z
 =\sum_n a_n^*w_n
 \le \|a\|_{q'}\|w\|_q.
\]
Consequently the formal identity from \(c_{00}\), equipped with the
\(\ell_{q'}\)-norm, into \(Z\) extends to a bounded operator
\(\ell_{q'}\to Z\).
Taking its adjoint, we obtain
\begin{equation}\label{lor:dual-coefficient-bound}
 \|(f(e_n))_n\|_q\le\|f\|\|w\|_q
 \qquad(f\in Z^*).
\end{equation}
For \(x\in E\), define \(f_x\in Z^*\) by
\[
 f_x(v)=(Jv)(x).
\]
Then \(\|f_x\|\le C_J\|x\|\). By
\eqref{lor:dual-coefficient-bound} and
\eqref{lor:embedding-bounds},
\[
 \|(y_n(x))_n\|_q
 \le c^{-1}C_J\|w\|_q\|x\|.
\]
This proves the first assertion, including all-point summability.

For the negative assertion we use precisely hypotheses
\textup{(H1)}--\textup{(H2)}.  Hypothesis \textup{(H1)} allows us to
apply Corollary~\ref{pre:union-allpoint}, while \textup{(H2)} converts
the weak-star nullity obtained from Lemma~\ref{pre:elementary}(ii) into
weak nullity inside $Y$.

Let \(z=(z_n)\subset S(Y)\), and suppose that \(D_{F_w}(z)\)
is dense in \(E\). Corollary~\ref{pre:union-allpoint} gives
\(G_q(z)=E\) for some \(q\in F_w\). Necessarily \(q>1\),
since otherwise \(z\) would be strongly normal.
By Lemma~\ref{pre:elementary}(ii), the sequence \(z\) is weak-star
null, hence weakly null in \(Y\).
Put \(v_n=J^{-1}z_n\). By \eqref{lor:embedding-bounds},
\[
 C_J^{-1}\le\|v_n\|_Z\le c^{-1}.
\]
Thus \(\widehat v_n=v_n/\|v_n\|_Z\) is normalized and weakly
null, and bounded normalization gives \(G_q((J\widehat v_n))=E\).
The gliding-hump principle \eqref{pre:gliding} supplies a
subsequence and normalized successive blocks \(b_k\) of the
basis of \(Z\) such that
\[
 \|\widehat v_{n_k}-b_k\|_Z\le2^{-k^2}.
\]
Lemma~\ref{pre:perturb} implies \(G_q((Jb_k))=E\).
By \eqref{pre:upper} and the lower bound for \(J\), there is
\(C_1<\infty\) such that
\[
 \left\|\sum_k a_kb_k\right\|_Z
 \le C_1\|a\|_{q'}\qquad(a\in c_{00}).
\]
The sequence \(b_k\) is weakly null because
$\|b_k-\widehat v_{n_k}\|_Z\to0$ and $(\widehat v_{n_k})$ is weakly
null. It therefore has no subsequence equivalent to the \(\ell_1\) basis: such an
isomorphism would give a bounded functional taking the value
one on that subsequence, contrary to weak nullity.
The classical domination estimate
\eqref{lor:classical-domination} now gives
\(\|a\|_Z\le C_2\|a\|_{q'}\) for finite coefficient sequences.
Apply this to \(a_i=w_i^{q-1}\), \(1\le i\le m\). Since
these coefficients are nonincreasing,
\begin{equation}\label{lor:amplification-ratio}
 \frac{\|\sum_{i=1}^{m}w_i^{q-1}e_i\|_Z}
      {\|(w_i^{q-1})_{i=1}^{m}\|_{q'}}
 =\left(\sum_{i=1}^{m}w_i^q\right)^{1/q}
 \longrightarrow\infty,
\end{equation}
contrary to the upper estimate, because \(q\in F_w\).
This proves \eqref{lor:forbidden-union}.

Finally, every subsequence of \(y\) has all-point \(q\)-summability
for \(q\in A_w\), while \eqref{lor:forbidden-union} excludes
every other exponent, both for \(y\) and for its subsequences.
This proves \eqref{lor:weight-spectrum} and the third assertion.
\end{proof}

\begin{remark}\label{lor:pitt-route}
For the power weight $w_n=n^{-1/r}$, $r>1$, put $s=r'$. Then
$d(w,1)$ is the classical Lorentz space $\ell_{s,1}$ up to equivalent
norms. L\'opez Molina's theorem \cite[Theorem~4.6(4)]{pitt} applies to
operators $T:\ell_{u,v}(\nu)\to\ell_{p,q}(\mu)$, with
$0<u,v,p,q<\infty$ and the measures specified there.
For counting measure on both copies of $\N$, its measure hypotheses hold,
and its compactness criterion is
$v>q$, $v\ge p$, $u\ge p$, and $u\ge q$.
For an operator $\ell_s=\ell_{s,s}\to\ell_{s,1}$ we use
$(u,v,p,q)=(s,s,s,1)$; all conditions hold since $1<s<\infty$.
Thus every such operator is compact. A seminormalized sequence in
$d(w,1)$ cannot satisfy an upper $\ell_s$ estimate, even on a tail,
since the associated operator would send the weakly null unit vectors
of $\ell_s$ to a seminormalized sequence.
Consequently, under the hypotheses of Theorem~\ref{lor:weight},
$G_r(z)\ne E$ for every normalized $z\subset Y$.
Proposition~\ref{pre:amplification} and (H1) exclude density as well.
This is an independent endpoint argument for the power-weight case;
the weight proof also treats arbitrary weights, endpoint one, and the
entire forbidden union.
\end{remark}

\begin{corollary}\label{lor:nonseparable}
Under the hypotheses of Theorem~\ref{lor:weight}, every normalized
weak-star null sequence \(z\subset Y\) satisfies
\[
 E/\overline{D_{F_w}(z)}\text{ is nonseparable}.
\]
\end{corollary}
\begin{proof}
Apply Corollary~\ref{pre:nonsep} with \(A=F_w\) and use
the hereditary forbidden-union assertion of
Theorem~\ref{lor:weight}.
\end{proof}

\begin{remark}\label{lor:nonisometric-symmetry}
An arbitrary isomorphic embedding \(J\) need not preserve
\(1\)-symmetry in the inherited norm of \(E^*\).
Theorem~\ref{lor:weight} requires no such preservation for its
summability assertions or its hereditary obstruction.
The additional symmetry hypothesis is used only when an isometric
disjoint-copy construction is desired.
\end{remark}

\subsection{Realization and classification of selection spectra}

\begin{corollary}[Open and closed hereditary thresholds]\label{lor:examples}
Each of the sets
\begin{equation}\label{lor:realized-list}
 (1,\infty),\qquad
 (r,\infty)\ (r>1),\qquad
 [r,\infty)\ (r>1),\qquad
 \varnothing
\end{equation}
is the spectrum of a normalized weakly null \(1\)-symmetric Schauder
basic sequence \(y\subset E^*\).
Every subsequence has the same spectrum.
Moreover, if \(I\) is the spectrum in question and \(Y=[y_n]\), then
\begin{equation}\label{lor:example-union}
 \overline{\bigcup_{q\in(0,\infty)\setminus I}G_q(z)}\ne E
 \qquad\text{for every normalized sequence }z\subset Y.
\end{equation}
The space \(Y\) can be chosen hereditary \(\ell_1\).
\end{corollary}

\begin{proof}
Apply the canonical case of Theorem~\ref{lor:weight}.
For the four cases in \eqref{lor:realized-list}, choose, respectively,
\[
 w_n=\frac1n,\qquad
 w_n=n^{-1/r},\qquad
 w_n=n^{-1/r}
       \left(\frac{\log 2}{\log(n+1)}\right)^{2/r},\qquad
 w_n=\frac{\log 2}{\log(n+1)}.
\]
Every weight satisfies \eqref{lor:weight-assumptions}.
The first two scalar spectra are \((1,\infty)\) and \((r,\infty)\).
For the third weight,
\[
 w_n^r=\frac{(\log 2)^2}{n(\log(n+1))^2},
\]
so the \(r\)-series converges. For \(q<r\), the exponent \(q/r<1\),
and the corresponding logarithmic correction does not make
\(\sum_n w_n^q\) converge; for \(q>r\) the series converges.
Thus this spectrum is \([r,\infty)\).
For the last weight, \(\sum_nw_n^q=\infty\) for every finite \(q>0\),
since a fixed positive power of \(\log(n+1)\) grows more slowly
than every positive power of \(n\).
All remaining conclusions follow from Theorem~\ref{lor:weight}.
\end{proof}

Recall that \(\cS(y)\) consists of those \(q>0\) for which a normalized
normal sequence \(y\) has a \(q\)-strongly normal subsequence,
and that \(\cB(y)\) denotes the corresponding exponents
for normalized block sequences when \(y\) is Schauder basic.
The next result determines their possible values.

In particular, the weight $w_n=\log 2/\log(n+1)$ in
Corollary~\ref{lor:examples} gives a normalized weakly null basic sequence
with no $p$-strongly normal subsequence for any $p>0$. This supplies an
independent negative answer to the subsequence question in
\cite{sliwa2012} and \cite[Problem~7]{survey2018}; compare the earlier
constructions in \cite{kakol2026,sliwa2026}. No priority for the negative
answer itself is claimed here.

\begin{theorem}[Classification of selection spectra]\label{lor:classification}
As the Banach space and the normalized normal sequence vary, the
possible values of \(\cS(y)\) are exactly
\begin{equation}\label{lor:selection-list}
 \varnothing,\qquad
 (0,\infty),\qquad
 (1,\infty),\qquad
 [r,\infty),\qquad
 (r,\infty)\quad(r>1).
\end{equation}
The possible values of \(\cB(y)\), for normalized normal Schauder
basic sequences, are the same.
Every set in \eqref{lor:selection-list} can be realized by a normalized
weakly null \(1\)-symmetric Schauder basic sequence satisfying
\[
 \cP(y)=\cS(y)=\cB(y).
\]
For every set other than \((0,\infty)\), the realization can also
satisfy the hereditary forbidden-union conclusion of
Corollary~\ref{lor:examples}.
The lists are unchanged if weak-star nullity of the original normalized
sequence is not assumed (with basicness still required for $\cB$).
\end{theorem}

\begin{proof}
Both selection spectra are upper subsets of \((0,\infty)\).
Suppose first that \(q\in\cS(y)\) for some \(q\le1\).
The selected subsequence is strongly normal because
\(\ell_q\subset\ell_1\).  Apply the exponent-one selection theorem
\eqref{pre:selection} and then Lemma~\ref{pre:allp}; the resulting
further subsequence is \(t\)-strongly normal for every \(t>0\).
Hence \(\cS(y)=(0,\infty)\).
The same argument applies to \(\cB(y)\): first choose the normalized
\(q\)-strongly normal block sequence, then apply
\eqref{pre:selection}, and finally Lemma~\ref{pre:allp}.
A subsequence of a block sequence is again a block sequence of the
original basis.
Thus either selection spectrum, if it meets \((0,1]\), must be
\((0,\infty)\).

A nonempty upper interval disjoint from \((0,1]\) is either
\((1,\infty)\), or \([r,\infty)\), or \((r,\infty)\) for some \(r>1\).
This proves that no additional sets can occur.

For every set in \eqref{lor:selection-list} other than \((0,\infty)\),
use Corollary~\ref{lor:examples}.
The original sequence supplies all exponents in its spectrum,
both as a subsequence and as a block sequence with singleton supports.
Every normalized block belongs to \(Y=[y_n]\), so the hereditary
obstruction excludes all exponents outside that spectrum.
For \((0,\infty)\), use an orthonormal basis of \(\ell_2\):
its scalar evaluations have finite support on the dense subspace
\(c_{00}\), and hence it is \(t\)-strongly normal for every \(t>0\).
\end{proof}

\begin{remark}\label{lor:selection-interpretation}
Theorem~\ref{lor:classification} classifies the existence of good
subsequences and good normalized blocks.
The hereditary assertion excludes forbidden exponents from every
normalized sequence in the closed span.
It does not assert that every normalized block has all the allowed
exponents; an arbitrary normalized block sequence may fail even to be
normal. The original sequence supplies the positive witnesses.
In particular, for the realization \(\cB(y)=(1,\infty)\) one has
\(\inf\cB(y)=1\), and this infimum is not attained.
\end{remark}

\begin{corollary}[The general block-recovery boundary]
\label{cor:p-strong-block-dichotomy}
Let \(p>0\).
\begin{enumerate}
\item If \(0<p\le1\), every normalized \(p\)-strongly normal sequence
in an arbitrary Banach dual is strongly normal and has a weak-star
Schauder basic subsequence which is \(q\)-strongly normal simultaneously
for every \(q>0\).
\item If \(p>1\), there are a Banach space \(E\) and a normalized
weakly null \(1\)-symmetric Schauder basic sequence \(y\subset E^*\)
such that
\[
 G_p(y)=E,\qquad \cP(y)=\cS(y)=\cB(y)=[p,\infty),
\]
and no normalized sequence in \([y_n]\) is \(q\)-strongly normal
for any \(0<q<p\).
\end{enumerate}
Thus \(p=1\) is the exact universal boundary, over all Banach spaces,
for recovering a normalized strongly normal successive block sequence,
or one good simultaneously for every positive exponent.
\end{corollary}
\begin{proof}
For \(p\le1\), use \(\ell_p\subset\ell_1\),
\eqref{pre:selection}, and Lemma~\ref{pre:allp}. A subsequence is a
successive block sequence with singleton supports.
For \(p>1\), take the closed-threshold Lorentz model with spectrum
\([p,\infty)\) in Corollary~\ref{lor:examples}.
Theorem~\ref{lor:weight} gives all-point summability at \(p\), and
the hereditary forbidden-union conclusion excludes every exponent
below \(p\) from all normalized sequences in the closed span.
Theorem~\ref{lor:classification} gives the asserted selection spectra.
\end{proof}

\subsection{Spaceability of exact open thresholds}

For the power weights, a tensor estimate permits us to impose the
same open spectrum on every nonzero element of a closed space of
sequences.  We use ``spaceable'' in the standard sense; for background
on lineability and spaceability see \cite{bernalspaceability}.

\begin{lemma}[Tensor estimate for the power weight]\label{lor:tensor}
Let \(r\ge1\), put \(w_n=n^{-1/r}\), and let \(Z=d(w,1)\).
After identifying \(\mathbb N\times\mathbb N\) with \(\mathbb N\)
by any bijection, the coordinate tensor
\(a\otimes b=(a_ib_j)_{i,j}\) satisfies
\begin{equation}\label{lor:tensor-bound}
 \|a\otimes b\|_Z\le\|a\|_Z\|b\|_Z
 \qquad(a,b\in Z).
\end{equation}
The bound has constant one, independently of \(r\).
\end{lemma}

\begin{proof}
We first consider nonnegative finitely supported sequences, reordered
so that they are nonincreasing.
For \(m\) copies of \(b\), the coefficient multiplying \(b_k\)
in the Lorentz norm is
\[
 \sum_{i=1}^{m}((k-1)m+i)^{-1/r}.
\]
For \(1\le i\le m\),
\[
 (k-1)m+i-ki=(k-1)(m-i)\ge0.
\]
Consequently, with \(W(m)=\sum_{i=1}^{m}w_i\),
\[
 \sum_{i=1}^{m}((k-1)m+i)^{-1/r}
 \le k^{-1/r}\sum_{i=1}^{m}i^{-1/r}
 =w_kW(m).
\]
It follows that
\begin{equation}\label{lor:repetition-bound}
 \|\mathbf1_{\{1,\ldots,m\}}\otimes b\|_Z
 \le W(m)\|b\|_Z.
\end{equation}
Write a finite nonincreasing \(a\), with \(a_{M+1}=0\), as
\[
 a=\sum_{m=1}^{M}(a_m-a_{m+1})
                   \mathbf1_{\{1,\ldots,m\}}.
\]
The triangle inequality and \eqref{lor:repetition-bound} give
\[
 \|a\otimes b\|_Z
 \le \|b\|_Z
       \sum_{m=1}^{M}(a_m-a_{m+1})W(m)
 =\|a\|_Z\|b\|_Z.
\]
Changes of signs and permutations prove the same inequality for
arbitrary finite sequences.

Since \(c_{00}\) is dense in \(Z\), this bounded bilinear mapping
extends to \(Z\times Z\).
For clarity, if \(a^{(m)}\to a\) and \(b^{(m)}\to b\) are finite
approximations, the bound just proved gives
\begin{align*}
 \|a^{(m)}\otimes b^{(m)}
       -a^{(k)}\otimes b^{(k)}\|_Z
 &\le
 \|a^{(m)}-a^{(k)}\|_Z\|b^{(m)}\|_Z\\
 &\quad+\|a^{(k)}\|_Z\|b^{(m)}-b^{(k)}\|_Z.
\end{align*}
The tensors are therefore Cauchy.
Continuity of every coordinate functional identifies the limit
with \((a_ib_j)_{i,j}\), proving
\eqref{lor:tensor-bound}.
The choice of bijection has no effect on the norm by symmetry.
Finally, equality for \(a=b=e_1\) shows that the constant cannot
be decreased.
\end{proof}

\begin{theorem}[Spaceability of an exact open spectrum]\label{lor:open-space}
Let $r\ge1$, put $w_n=n^{-1/r}$ and $Z=d(w,1)$, and let
$J:Z\longrightarrow E^*$ satisfy hypotheses \textup{(H1)}--\textup{(H2)}
of Theorem~\ref{lor:weight}.
Suppose also that \(a\mapsto\|Ja\|\) is invariant under changes of
signs and permutations of coordinates. Put \(Y=JZ\).
Then there is a closed subspace
\[
 W\subset\ell_\infty(E^*)
\]
isometric to \(Y\), and therefore hereditary \(\ell_1\), such that
every nonzero \(u=(u_n)\in W\) has the following properties:
\begin{enumerate}
\item \(\|u_n\|\) is independent of \(n\) and is strictly positive;
\item \(u_n\to0\) weakly in \(E^*\);
\item \((u_n/\|u_1\|)\) is a normalized \(1\)-symmetric Schauder
      basic sequence;
\item \(G_q(u)=E\) for every \(q>r\), and
      \(\cP(u)=(r,\infty)\);
\item for every normalized sequence \(z\) in the closed span
      \([u_n]\), one has \(\overline{G_r(z)}\ne E\).
\end{enumerate}
In fact, the last assertion holds for every normalized sequence in
the common space \(Y\).
Moreover,
\[
 \cS(u/\|u_1\|)=\cB(u/\|u_1\|)=(r,\infty).
\]
Thus \(\inf\cB(u/\|u_1\|)=r\), and this infimum is not attained.
For every \(r\ge1\) these conclusions hold, in particular, for the
canonical choice \(E=Z^*\), \(J=J_Z\); in that case \(W\) is
isometric to \(Z\).
\end{theorem}

\begin{proof}
Partition \(\mathbb N\) into pairwise disjoint infinite sets
\[
 A_n=\{a_{n,1}<a_{n,2}<\cdots\},\qquad n\in\mathbb N.
\]
For \(v=\sum_kv_ke_k\in Z\), define
\[
 R_nv=\sum_{k=1}^{\infty}v_ke_{a_{n,k}}.
\]
Since the basis of \(Z\) is \(1\)-symmetric, each \(R_n\) is a
linear isometry on \(Z\).
To verify this without using an infinite permutation, first apply
symmetry to each finite truncation of \(v\): the relevant finite
injection extends to a permutation of \(\mathbb N\).
Then pass to the limit using the Schauder expansion.
The same argument, using the assumed invariance of the norm induced
by \(J\), shows that
\begin{equation}\label{lor:copies-induced-isometry}
 \|JR_nv\|=\|Jv\|
 \qquad(n\in\mathbb N,\ v\in Z).
\end{equation}

Define
\[
 \Theta:Y\longrightarrow\ell_\infty(E^*),
 \qquad
 \Theta(Jv)=(JR_nv)_{n=1}^{\infty}.
\]
This is well defined and linear, and
\eqref{lor:copies-induced-isometry} shows that it is an isometry.
Its range \(W=\Theta(Y)\) is therefore closed and isometric to
\(Y\). The latter is isomorphic to the hereditary \(\ell_1\)
space \(Z\). Formula \eqref{lor:copies-induced-isometry} also proves
the assertion about the coordinate norms.

Fix \(v\in Z\).
For \(a\in c_{00}\), the vector
\[
 \sum_n a_nR_nv
\]
has, up to a fixed reindexing, the coordinates
\((a_nv_k)_{n,k}\). By Lemma~\ref{lor:tensor},
\[
 \left\|\sum_n a_nR_nv\right\|_Z
 \le\|a\|_Z\|v\|_Z.
\]
Consequently
\[
 T_v:Z\longrightarrow Z,\qquad T_ve_n=R_nv,
\]
is a bounded operator, with \(\|T_v\|\le\|v\|_Z\).
For \(x\in E\), let \(f_x\in Z^*\) be as in the proof of
Theorem~\ref{lor:weight}. Then
\[
 (JR_nv)(x)=(f_x\circ T_v)(e_n).
\]
Here \(f_x\circ T_v\in Z^*\).
For every \(q>r\), the weight \((n^{-1/r})\) belongs to \(\ell_q\);
therefore \eqref{lor:dual-coefficient-bound} implies
\[
 ((JR_nv)(x))_n\in\ell_q.
\]
This proves \(G_q(\Theta(Jv))=E\) for \(q>r\).

For the same power weight, \(A_w=(r,\infty)\) and
\(F_w=(0,r]\).
Because \(G_q(z)\subset G_r(z)\) for \(q\le r\),
\[
 D_{F_w}(z)=G_r(z).
\]
Theorem~\ref{lor:weight} consequently gives
\(\overline{G_r(z)}\ne E\) for every normalized sequence
\(z\subset Y\).
If \(v\ne0\), all coordinates \(JR_nv\) have the same nonzero norm.
Normalization therefore does not change any of their
summability domains. No exponent \(q\le r\) belongs to the spectrum
of \(\Theta(Jv)\), proving the claimed exact spectrum and the
hereditary conclusion.

We verify weak nullity explicitly.
For a finite truncation \(v^{(m)}=\sum_{k=1}^{m}v_ke_k\) and
\(F\in E^{**}\),
\[
 F(JR_nv^{(m)})
 =\sum_{k=1}^{m}v_kF(Je_{a_{n,k}})
 \longrightarrow0.
\]
Indeed, \(Je_j\to0\) weakly in \(E^*\), and for each fixed \(k\)
the disjointness of the sets \(A_n\) implies
\(a_{n,k}\to\infty\).
Moreover,
\[
 \sup_n\|JR_n(v-v^{(m)})\|
 \le\|J\|\|v-v^{(m)}\|_Z\longrightarrow0.
\]
Finite-support approximation now proves \(JR_nv\to0\) weakly.

Assume \(v\ne0\).
The vectors \(JR_nv\) have disjoint nonempty coordinate supports
when pulled back to \(Z\), so they are linearly independent.
Permuting these vectors or changing their signs corresponds to
permuting or changing the signs of the coordinates of
\(\sum_n a_nR_nv\).
The invariance of the induced norm therefore proves \(1\)-symmetry
after normalization.
Suppressing any set of the supports \(A_n\) is contractive for this
norm. In particular, the partial-sum projections associated with
\((JR_nv)\) are contractive on its algebraic span.
They extend to its closed span, and their uniform boundedness,
together with density of the finite sums, implies convergence of
all partial sums there. Thus this sequence is Schauder basic.

Finally, the original normalized sequence supplies every exponent
\(q>r\) as both a subsequence and a normalized block sequence.
Every normalized block belongs to \(Y\), where all \(q\le r\) have
already been excluded. The same is true of every subsequence.
This proves the assertions about \(\cS\), \(\cB\), and the
unattained block infimum.
\end{proof}

\begin{remark}\label{lor:open-space-scope}
The canonical construction provides an exact open threshold for
every \(r\ge1\), including \(r=1\).
The embedding formulation separates the Lorentz argument from the
question of realizing its symmetric norm inside a specified dual.
In this paper the formulation is used for Lorentz copies inside the
larger direct sums and mixed spaces of Section~\ref{pert:section};
Lemma~\ref{pre:nostrong} verifies the hypotheses in those applications.
We do not claim a separate classification of noncanonical embeddings.
Hypotheses \textup{(H1)}--\textup{(H2)} are exactly the structural
properties needed for the weight criterion, while invariance of the
induced norm is used only for the isometric and constant-coordinate-norm
conclusions of Theorem~\ref{lor:open-space}.
\end{remark}

\subsection{The original-sequence spectrum without basicness}
The collapse at exponent one in Theorem~\ref{lor:classification} is a
selection phenomenon. Without a basicness assumption, the spectrum of
the original sequence has no corresponding restriction.

\begin{theorem}\label{orig:classification}
The possible values of $\mathcal P(y)$ for normal sequences in Banach duals
are exactly
\[
 \varnothing,\qquad (0,\infty),\qquad [r,\infty),\qquad (r,\infty)
 \quad(0<r<\infty).
\]
Every such set can be realized by a normalized weakly null sequence in the
real separable Hilbert space $H=\mathbb R e_0\oplus_2\ell_2$, identified with
its dual. The realization can be chosen with $\mathcal S(y)=(0,\infty)$.
No basicness is asserted.
\end{theorem}
\begin{proof}
Since $\ell_p\subset\ell_q$ whenever $0<p<q$, the set $\mathcal P(y)$ is
upward closed. An upward-closed subset of $(0,\infty)$ is empty, the whole
interval, or an interval with a positive finite left endpoint, which may
or may not be included. These are precisely the displayed possibilities.

Let $a=(a_n)$ be any nonnegative null sequence and define
\[
 y_{2n-1}=e_n,\qquad
 y_{2n}=\frac{e_n+a_ne_0}{\sqrt{1+a_n^2}}.
\]
Each vector has norm one. For $x=(t,b)\in\mathbb R\oplus_2\ell_2$,
\[
 y_{2n-1}(x)=b_n,\qquad
 y_{2n}(x)=\frac{b_n+t a_n}{\sqrt{1+a_n^2}},
\]
and both expressions tend to zero. Thus $y$ is weakly null. The normalization
factors and their reciprocals are bounded. For every $p>0$ we have
\begin{align*}
 x\in G_p(y)
 &\iff b\in\ell_p\text{ and }b+ta\in\ell_p\\
 &\iff b\in\ell_p\text{ and }ta\in\ell_p.
\end{align*}
Only linearity of $\ell_p$ is needed, also when $p<1$. Consequently
\[
 G_p(y)=
 \begin{cases}
  \mathbb R\oplus(\ell_2\cap\ell_p),&a\in\ell_p,\\
  \{0\}\oplus(\ell_2\cap\ell_p),&a\notin\ell_p.
 \end{cases}
\]
Since $c_{00}$ is dense in $\ell_2$, this proves
\begin{equation}\label{orig:domain-spectrum}
 \mathcal P(y)=\{p>0:a\in\ell_p\}.
\end{equation}
The following choices give every required set:
\begin{center}
\begin{tabular}{ll}
\toprule
$a_n$ & $\mathcal P(y)$\\
\midrule
$2^{-n}$ & $(0,\infty)$\\[2pt]
$1/\log(n+2)$ & $\varnothing$\\[2pt]
$(n+1)^{-1/r}$ & $(r,\infty)$\\[2pt]
$\bigl((n+1)(\log(n+2))^2\bigr)^{-1/r}$ & $[r,\infty)$\\
\bottomrule
\end{tabular}
\end{center}
For the last choice, the series of $r$th powers converges, whereas every
$p$th-power series with $p<r$ diverges. For the second choice, every finite
power series diverges. Finally, the odd subsequence is $(e_n)$ and its
$G_p$-domain contains $\mathbb R\oplus c_{00}$ for every $p>0$. This gives
$\mathcal S(y)=(0,\infty)$.
\end{proof}

\begin{remark}
These realizations are not Schauder basic. If two distinct indices have
$a_n\ne0$, the identities
\[
 \sqrt{1+a_n^2}\,y_{2n}-y_{2n-1}=a_ne_0
\]
give a finite linear dependence. Thus Theorem~\ref{orig:classification} does not
classify original-sequence spectra under a basicness requirement, and does
not classify triples $(\mathcal P,\mathcal S,\mathcal B)$ for one basic sequence.
In particular, $[1,\infty)$ and $(1/2,\infty)$ are realized without a basicness
assumption, but no corresponding basic realization is claimed here.
\end{remark}

\section{Perturbations and joint selection spectra}\label{pert:section}

The three spectra have different stability requirements.  For an
individual exponent the relevant error ideal is $\ell_p$, whereas
subsequence and block selection have the respective thresholds $c_0$
and $\ell_1$.  All assertions below concern the norm of the coordinate
errors; no equivalence assumption on the two basic sequences is needed.

\subsection{Three sharp error thresholds}

\begin{theorem}[Universal perturbation thresholds]\label{pert:thresholds}
Let $\varepsilon=(\varepsilon_n)$ be a sequence of nonnegative real
numbers, and fix $p>0$.  Consider pairs of normalized normal sequences
$y,z\subset E^*$ satisfying
\begin{equation}\label{pert:budget}
 \|y_n-z_n\|\le\varepsilon_n\qquad(n\in\N).
\end{equation}
Each of the following is an equivalence between membership of the
error sequence in the stated class and a universal preservation property.
\begin{enumerate}[label=\textup{(\roman*)}]
\item $\varepsilon\in\ell_p$ if and only if, for every Banach space
      $E$ and every pair satisfying \eqref{pert:budget},
      $p\in\cP(y)$ exactly when $p\in\cP(z)$.  If
      $\varepsilon\in\ell_p$, the stronger identity
      $G_p(y)=G_p(z)$ holds for every such pair.
\item $\varepsilon\in c_0$ if and only if, for every Banach space
      $E$ and every pair satisfying \eqref{pert:budget},
      $\cS(y)=\cS(z)$.
\item $\varepsilon\in\ell_1$ if and only if, for every Banach space
      $E$ and every pair of Schauder basic sequences satisfying
      \eqref{pert:budget}, $\cB(y)=\cB(z)$.
\end{enumerate}
Failure of \textup{(i)} can be witnessed by weakly null sequences in a
separable Hilbert space.  Failure of either \textup{(ii)} or
\textup{(iii)} can be witnessed by weakly null $1$-unconditional basic
sequences.  In the counterexample for \textup{(iii)}, the actual errors
can always be chosen to tend to zero.
\end{theorem}

\begin{proof}
\emph{Proof map.}
Sufficiency is transferred directly through the corresponding error
ideal.  For necessity, part~\textup{(i)} uses a rank-one Hilbert-space
perturbation, part~\textup{(ii)} inserts an $\ell_2$ component on a
nonvanishing error set, and part~\textup{(iii)} converts a nonsummable
error sequence into a Lorentz weight whose row averages create a
simultaneously good block sequence.

\emph{Sufficiency.}
Part (i) follows from Lemma~\ref{pre:perturb}.
Suppose next that $\varepsilon_n\to0$.  For each fixed $q>0$,
Corollary~\ref{pre:norm-null-subsequences} applied with $A=\{q\}$
shows that $y$ has a $q$-strongly normal subsequence exactly when $z$
does.  Hence $\cS(y)=\cS(z)$, proving (ii).

For (iii), suppose $\sum_n\varepsilon_n<\infty$, and let
$v_k=\sum_{n\in I_k}c_n^{(k)}y_n$ be normalized $q$-strongly normal
blocks of $y$.  The coordinate functionals of a normalized Schauder
basis are uniformly bounded, so $|c_n^{(k)}|\le C_y$ for a constant
independent of $n,k$.  For the corresponding blocks
$w_k=\sum_{n\in I_k}c_n^{(k)}z_n$, we have
\[
 \|v_k-w_k\|\le C_y\sum_{n\in I_k}\varepsilon_n\longrightarrow0.
\]
After a further selection these errors are at most $2^{-k^2}$, the
vectors $w_k$ are nonzero, and
\[
 \left\|v_k-\frac{w_k}{\|w_k\|}\right\|
 \le2\|v_k-w_k\|.
\]
Lemma~\ref{pre:perturb} now gives $q\in\cB(z)$.
The reverse inclusion follows by interchanging the two bases.

We use the following elementary observation in the converse
constructions.  If a nonnegative sequence $b$ does not belong to
$\ell_s$, where $s>0$, it has a minorant
\begin{equation}\label{pert:minorant}
 0\le a_n\le\min\{b_n,1\},\qquad a\in c_0\setminus\ell_s.
\end{equation}
Indeed, for each $k$, the series
$\sum_n\min\{b_n,1/k\}^s$ diverges.  Choose successive finite
intervals $I_k$ exhausting $\N$ such that its sum over $I_k$ is at
least one, and set $a_n=\min\{b_n,1/k\}$ on $I_k$.

\emph{Necessity in (i).}
Suppose $\varepsilon\notin\ell_p$.
Partition $\N$ into two infinite sets
$\{i_k:k\in\N\}$ and $\{j_k:k\in\N\}$ such that
$\sum_k\varepsilon_{j_k}^p=\infty$.
Such a partition is obtained, for example, by alternating successive
finite intervals on which the divergent series has sum at least one.
By \eqref{pert:minorant}, choose
$0\le a_k\le\min\{\varepsilon_{j_k},1\}$ with
$a\in c_0\setminus\ell_p$.
In $H=\ell_2(\N_0)$, identified with its dual, put
\[
 y_{i_k}=y_{j_k}=e_k,\qquad
 z_{i_k}=e_k,\qquad
 z_{j_k}=\frac{e_k+a_ke_0}{\sqrt{1+a_k^2}}.
\]
Both sequences are normalized and weakly null.  Moreover,
\[
 \|z_{j_k}-y_{j_k}\|
 =\left(2-\frac{2}{\sqrt{1+a_k^2}}\right)^{1/2}\le a_k,
\]
so \eqref{pert:budget} holds.  The paired coordinates give
\[
 G_p(y)=\R e_0\oplus(\ell_2\cap\ell_p),\qquad
 G_p(z)=\{0\}\oplus(\ell_2\cap\ell_p).
\]
For the second identity, the coordinates indexed by $i_k$ first
give $(x_k)\in\ell_p$, and those indexed by $j_k$ then force
$x_0a\in\ell_p$, hence $x_0=0$.
Thus $G_p(y)$ is dense whereas $G_p(z)$ has closure $e_0^\perp$.

\emph{Necessity in (ii).}
Suppose $\varepsilon\notin c_0$.
There are an infinite set $J\subset\N$ and $0<\delta\le1$ such that
$\varepsilon_n\ge\delta$ on $J$.
Let $Z=d((n^{-1}),1)$ and
$E=Z^*\oplus_1\ell_2$, so $E^*=Z^{**}\oplus_\infty\ell_2$.
Put
\[
 y_n=(J_Ze_n,0),\qquad
 z_n=(J_Ze_n,\delta\mathbf1_J(n)e_n).
\]
These are normalized weakly null $1$-unconditional basic sequences:
their coefficient norms are, respectively,
$\|c\|_Z$ and
$\max\{\|c\|_Z,\delta\|(c_n)_{n\in J}\|_2\}$, so all coordinate
projections are contractive.  The error bound holds, and $\cS(y)=(1,\infty)$ follows from the canonical Lorentz criterion.
Indeed, for every sequence $u\subset J_ZZ$,
\[
 G_q^{E}((u_n,0))=G_q^{Z^*}(u)\times\ell_2
\]
by Lemma~\ref{pre:directsum}. Lemma~\ref{pre:nostrong} therefore transfers
(H1) to $J_ZZ\oplus\{0\}$, and its relative weak-star nullity is precisely
weak nullity in $Z$, giving (H2). The same domain identity applies to every
subsequence.
For finitely supported coefficients on $J$, Cauchy--Schwarz gives
\[
 \delta\|c\|_2
 \le\left\|\sum_{n\in J}c_nz_n\right\|
 \le\max\{\|(n^{-1})\|_2,\delta\}\,\|c\|_2.
\]
Thus $(z_n)_{n\in J}$ is equivalent to the basis of $\ell_2$.  Its
norm-closed span is therefore reflexive; by Lemma~\ref{pre:reflexive}
it is weak-star closed and the displayed basis is weak-star Schauder
basic.  Lemma~\ref{pre:allp} then gives $\cS(z)=(0,\infty)$.

\emph{Necessity in (iii).}
Suppose $\varepsilon\notin\ell_1$.
Choose $a\in c_0\setminus\ell_1$ satisfying
$0\le a_n\le\min\{\varepsilon_n,1\}$ by
\eqref{pert:minorant}.  Since $a\notin\ell_1$ and $a\ge0$, it has infinitely many positive
coordinates; hence $a_m^*>0$ for every $m$.  For its nonincreasing
rearrangement $a^*$, define
\[
 A(m)=\sum_{i=1}^m a_i^*,\quad A(0)=0,\qquad
 W(m)=\sqrt{A(m)/A(1)},\qquad w_m=W(m)-W(m-1).
\]
Concavity of $A$ and of the square root shows that $(w_m)$ is
positive and nonincreasing, with $w_1=1$.
Moreover $W(m)\to\infty$.  Since
$a_i^*\le a_1^*=A(1)$, one has $A(m)\le mA(1)$, and therefore
\[
 w_m\le \frac{W(m)}m\le m^{-1/2}.
\]
Thus $w$ satisfies \eqref{lor:weight-assumptions}.
Take $Z=d(w,1)$ and $E=Z^*\oplus_1c_0$, so that
$E^*=Z^{**}\oplus_\infty\ell_1$, and define
\begin{equation}\label{pert:lorentz-pair}
 y_n=(J_Ze_n,0),\qquad z_n=(J_Ze_n,a_ne_n).
\end{equation}
Both sequences are normalized and weakly null, since $e_n$ is
weakly null in $Z$ and $a_ne_n\to0$ in the norm of $\ell_1$.
They are $1$-unconditional basic: the coefficient norm for $z$ is
\begin{equation}\label{pert:coefficient-norm}
 \left\|\sum_nc_nz_n\right\|
 =\max\left\{\|c\|_Z,\sum_na_n|c_n|\right\},
\end{equation}
and coordinate projections are contractive in both terms.
For clarity, bounded coordinate projections on the finite span
extend to its closure and converge to the identity there, which
also proves the Schauder property.
The errors are $a_n\le\varepsilon_n$ and tend to zero.
The Lorentz criterion gives
\begin{equation}\label{pert:old-spectrum}
 \cP(y)=\cS(y)=\cB(y)=A_w\subset(1,\infty).
\end{equation}
Indeed, Lemma~\ref{pre:directsum} identifies every summability domain
of a sequence in the first factor with the corresponding canonical domain
times $c_0$; thus all uses of the Lorentz criterion in this direct sum
reduce literally to the canonical case.

We construct successive finite sets $F_k$ such that, with
$m_k=|F_k|$ and $D_k=\sum_{n\in F_k}a_n$,
\begin{equation}\label{pert:large-mass}
 D_k\ge2^{k^2}W(m_k).
\end{equation}
To choose a set beyond an arbitrary cutoff $N$, put
$C_N=\sum_{n\le N}a_n$.
For every $m$, a set of $m$ indices beyond $N$ can be chosen with
sum at least $A(m)-C_N$: start with the indices of the $m$ largest
terms, remove those at most $N$, and add arbitrary tail indices
until the cardinality is again $m$.
Since
\[
 \frac{A(m)-C_N}{W(m)}\longrightarrow\infty,
\]
the required successive choice is possible.
Now put
\[
 v_k=\frac1{D_k}\sum_{n\in F_k}z_n,
 \qquad
 u_k=\left(0,\frac1{D_k}\sum_{n\in F_k}a_ne_n\right).
\]
Equations \eqref{pert:coefficient-norm} and
\eqref{pert:large-mass} give
\[
 \|v_k\|=\|u_k\|=1,\qquad
 \|v_k-u_k\|=\frac{W(m_k)}{D_k}\le2^{-k^2}.
\]
The sets $F_k$ are contained in successive finite intervals, so
$v$ is a normalized block sequence of $z$.
Evaluations of $u$ have finite support on the dense subspace
$Z^*\oplus c_{00}$ of $E$.
Consequently $\cP(u)=(0,\infty)$, and
Lemma~\ref{pre:perturb} gives $\cP(v)=(0,\infty)$.
In particular $\cB(z)=(0,\infty)\ne\cB(y)$.
These normalized blocks are weak-star null, with a disjoint
$\ell_1$ component of norm one.
This proves all three converses.
\end{proof}

\begin{corollary}[Perturbations of the whole original-sequence spectrum]
\label{pert:wholeP}
Let $\varepsilon=(\varepsilon_n)$ be nonnegative.  The following are
equivalent:
\begin{enumerate}[label=\textup{(\roman*)}]
\item $\varepsilon\in\bigcap_{p>0}\ell_p$;
\item for every Banach space $E$ and every pair of normalized normal
sequences $y,z\subset E^*$ satisfying \eqref{pert:budget}, one has
$\cP(y)=\cP(z)$.
\end{enumerate}
If \textup{(i)} holds, then $G_p(y)=G_p(z)$ for every $p>0$ and every
such pair.
\end{corollary}
\begin{proof}
The implication \textup{(i)}$\Rightarrow$\textup{(ii)}, including the
domain identities, follows from Theorem~\ref{pert:thresholds}(i) for
every $p>0$.  Conversely, if
$\varepsilon\notin\bigcap_{p>0}\ell_p$, choose $p>0$ with
$\varepsilon\notin\ell_p$.  The sharpness part of
Theorem~\ref{pert:thresholds}(i) gives a pair for which membership of
$p$ in $\cP$ is not preserved, so the two spectra are different.
\end{proof}

The block counterexample also separates the individual and
subsequence spectra from the block spectrum.  The separation is
available for every proper spectrum in
Theorem~\ref{lor:classification}.

\Needspace{10\baselineskip}
\begin{corollary}[Complete block recovery without subsequence recovery]
\label{pert:separation}
Let
\[
 I\in\{\varnothing,(1,\infty),(r,\infty),[r,\infty):r>1\}.
\]
There are a Banach space $E$ and normalized weakly null
$1$-unconditional basic sequences $y,z\subset E^*$ such that
$\|y_n-z_n\|\to0$ and
\[
 \cP(y)=\cS(y)=\cB(y)=\cP(z)=\cS(z)=I,
 \qquad \cB(z)=(0,\infty).
\]
In fact one normalized block sequence of $z$ is $q$-strongly
normal simultaneously for every $q>0$.
\end{corollary}

\begin{proof}
Choose the Lorentz weight $w$ realizing $I$ in
Corollary~\ref{lor:examples}, and put $W(m)=\sum_{i=1}^mw_i$.
Since $W(m)/m\to0$, choose positive integers $m_k$ such that
\[
 \frac{W(m_k)}{2^{-k}m_k}\le2^{-k^2}.
\]
Partition $\N$ into successive intervals $F_k$ of lengths $m_k$,
and set $a_n=2^{-k}$ on $F_k$.
Use \eqref{pert:lorentz-pair} in $E=Z^*\oplus_1c_0$.
The preceding proof establishes normalization, weak nullity,
unconditional basicness, \eqref{pert:old-spectrum}, and the
existence of a normalized block sequence $v$ with
$\cP(v)=(0,\infty)$.
Since $a_n\to0$, part (ii) of Theorem~\ref{pert:thresholds}
gives $\cS(z)=I$.
For $q\in I=A_w$, the canonical evaluation sequence belongs to
$\ell_q$ at every point of $Z^*$, and the added term has finite
support at every point of $c_{00}$.
Thus $G_q(z)$ contains the dense subspace $Z^*\oplus c_{00}$.
It follows that $I\subset\cP(z)\subset\cS(z)=I$, including when
$I$ is empty.
\end{proof}

\begin{remark}\label{pert:quotient}
Theorem~\ref{pert:thresholds}(ii) says that the subsequence
spectrum is an invariant of normalized normal sequences modulo
$c_0(E^*)$ in $\ell_\infty(E^*)$.
Neither the individual spectrum nor, for basic sequences, the
block spectrum has this invariance.
The latter failure persists for normalized weakly null
$1$-unconditional basic sequences with identical individual and
subsequence spectra, by Corollary~\ref{pert:separation}.
\end{remark}

\subsection{Joint realization of subsequence and block spectra}
\label{joint:section}

The individual classification of selection spectra has a joint form.
Write
\[
 \mathfrak A=
 \{\varnothing,(0,\infty),(1,\infty)\}
 \cup\{[r,\infty),(r,\infty):r>1\}.
\]
Every subsequence of a normalized basic sequence is a normalized
block sequence, so $\cS(y)\subset\cB(y)$.  The next theorem shows
that this inclusion is the only additional restriction.

\begin{theorem}[Joint selection spectra]\label{joint:classification}
The pairs $(\cS(y),\cB(y))$ arising from normalized Schauder basic
sequences in Banach duals are exactly
\[
 \{(I,J): I,J\in\mathfrak A,\ I\subseteq J\}.
\]
Every such pair is realized by a normalized, weakly null,
$1$-unconditional basic sequence satisfying $\cP(y)=\varnothing$.
If $I=J$, there is also a normalized weakly null $1$-symmetric
realization satisfying $\cP(y)=\cS(y)=\cB(y)=I$.
\end{theorem}

We give the elementary saturation facts used in the construction.
A Banach space is called $\ell_1$-saturated if every
infinite-dimensional closed subspace contains a copy of $\ell_1$;
this is the hereditary $\ell_1$ property used above.

\begin{lemma}\label{joint:saturation}
Let $Z$ have a normalized $1$-unconditional Schauder basis and be
$\ell_1$-saturated.  For positive integers $m_k$, let
\[
 X=\Bigl(\bigoplus_{k=1}^{\infty}\ell_1^{m_k}\Bigr)_Z,
 \qquad
 \|(a_{k,i})\|_X
 =\Bigl\|\sum_k\Bigl(\sum_{i=1}^{m_k}|a_{k,i}|\Bigr)e_k\Bigr\|_Z,
\]
where $X$ is the completion of the finitely supported arrays.
Then $X$ is $\ell_1$-saturated.  A finite direct sum of
$\ell_1$-saturated spaces is again $\ell_1$-saturated.
\end{lemma}

\begin{proof}
The coordinate spaces of $X$ form a finite-dimensional
$1$-unconditional decomposition.  Given an infinite-dimensional
closed $M\subset X$, the usual finite-codimensional selection
chooses normalized $u_j\in M$ and normalized successive
decomposition blocks $v_j$ with
$\sum_j\|u_j-v_j\|<1/2$: choose $u_j$ annihilating all coordinates
up to the preceding support and then truncate its tail.
The sequence $(v_j)$ is $1$-unconditional.  Its positive norm
profiles
\[
 b_j=\sum_k\|v_j^k\|_1e_k\in Z
\]
are normalized and have disjoint supports, and
\[
 \Bigl\|\sum_ja_jv_j\Bigr\|_X
 =\Bigl\|\sum_ja_jb_j\Bigr\|_Z
 \qquad((a_j)\in c_{00}).
\]
Thus $[v_j]$ is isometric to a subspace of $Z$ and contains
$\ell_1$. To justify the perturbation step directly, put
$\rho=\sum_j\|u_j-v_j\|<1/2$. Since $(v_j)$ is normalized and
$1$-unconditional, $|a_j|\leq\|\sum_i a_iv_i\|_X$ for every finite
coefficient family. Therefore
\[
 \Bigl\|\sum_j a_j(u_j-v_j)\Bigr\|_X
 \leq\rho\Bigl\|\sum_j a_jv_j\Bigr\|_X.
\]
The map $v_j\mapsto u_j$ extends to an isomorphism from $[v_j]$ onto
$[u_j]$, with lower bound $1-\rho>0$. Thus $[u_j]\subset M$ contains
$\ell_1$.

For the second assertion it suffices to consider $A\oplus_\infty B$.
Let $M$ be an infinite-dimensional closed subspace.  If the first
coordinate projection is bounded below on an infinite-dimensional
subspace of $M$, saturation of $A$ gives the conclusion.
Otherwise that projection, denoted by $P_A$, is strictly singular.
Choose inductively normalized $u_j\in M$ such that
$\|P_Au_j\|<2^{-j-4}$. At each stage also require the later vectors
to annihilate a finite set of functionals that norms the preceding
finite-dimensional span within a factor two. At each later step retain
all the finitely many norming sets chosen at earlier steps. Such requirements have
finite codimension, and strict singularity ensures the existence of
a unit vector with arbitrarily small $P_A$-image in each resulting
infinite-dimensional subspace. The partial-sum projections of
$(u_j)$ consequently have norm at most two, and its coordinate
functionals have norm at most four. For a finite sum $u=\sum_ja_ju_j$,
\[
 \|P_Au\|\le\sum_j|a_j|\,\|P_Au_j\|
 \le4\|u\|\sum_j2^{-j-4}<\tfrac12\|u\|.
\]
The same estimate holds on $N=[u_j]$.
Since the ambient norm is the maximum of the two coordinate norms,
the second projection is bounded below on $N$.  Saturation of $B$
completes the proof.
\end{proof}

\begin{proof}[Proof of Theorem~\ref{joint:classification}]
\emph{Proof map.}
Necessity follows from the one-dimensional shape theorem and the obvious
inclusion $\cS\subseteq\cB$.  For realization, two Lorentz models
control the allowed subsequence and block exponents, while a mixed
$\ell_1$-saturated space excludes every forbidden exponent.  The
remaining case $J=(0,\infty)$ is obtained by interleaving a sequence
with empty selection spectrum.

The necessity argument in Theorem~\ref{lor:classification} does
not require the original sequence to be weak-star null: both
selection spectra are upper sets, and any good selection at an
exponent at most one has a further subsequence good at every
positive exponent.  Thus both spectra belong to $\mathfrak A$,
and $\cS(y)\subset\cB(y)$.  For the separate $1$-symmetric
assertion when $I=J$, use the Lorentz models for proper $I$ and
an orthonormal basis for $I=(0,\infty)$.

We first construct the realization with empty $\cP$ when
\[
 I\subseteq J\subsetneq(0,\infty),
\]
allowing equality and the case $J=\varnothing$.

Choose Lorentz spaces $Z_I=d(w^I,1)$ and $Z_J=d(w^J,1)$ from
Theorem~\ref{lor:weight}, with allowed sets $I$ and $J$.
Their coordinate bases are normalized, weakly null and
$1$-symmetric, and both spaces are $\ell_1$-saturated.
Put $W_I(m)=\sum_{n=1}^m w_n^I$.  Since $w_n^I\to0$, one has
$W_I(m)/m\to0$.  With $\delta_k=2^{-k}$ choose positive integers
$m_k$ satisfying
\begin{equation}\label{joint:block-lengths}
 \frac{W_I(m_k)}{\delta_km_k}\leq2^{-k^2}.
\end{equation}
Form the space $X=(\bigoplus_k\ell_1^{m_k})_{Z_J}$ of
Lemma~\ref{joint:saturation}, and put
\[
 V=Z_I\oplus_\infty X,\qquad E=V^*.
\]
We regard $V$ as its canonical image in $E^*=V^{**}$.
By the lemma, $V$ is $\ell_1$-saturated; consequently this canonical
image contains no normalized strongly normal sequence by
Lemma~\ref{pre:nostrong}.

Enumerate the coordinate pairs $(k,i)$, $1\leq i\leq m_k$, in
successive order, writing $n=n(k,i)$.  In $V\subset E^*$ define
\[
 y_n=(e_n^I,0),\qquad
 z_n=(e_n^I,\delta_ke_{k,i}^X)
 \quad(n=n(k,i)).
\]
Both sequences are normalized and $1$-unconditional basic.
The first is weakly null, and
$\|z_n-y_n\|=\delta_k\to0$, so the second is weakly null as well.
In fact, $\cP(z)=\varnothing$.  Since $W_I(m_k)\geq1$,
\eqref{joint:block-lengths} forces $m_k\geq2^{k^2+k}$.
The sum over the $k$th row has norm exactly $\delta_km_k$.
For every fixed $q>1$, its ratio to the $\ell_{q'}$ norm of its
coefficient vector is
\begin{equation}\label{joint:empty-point-spectrum}
 \delta_km_k^{1/q}
 \geq 2^{(k^2+k)/q-k}\longrightarrow\infty.
\end{equation}
Thus $z$ has no upper $\ell_{q'}$ estimate and hence $G_q(z)\ne E$.
There remains only the possibility that $G_q(z)$ might be proper but
dense.  Proposition~\ref{pre:amplification} rules this out: such density
would produce a normalized strongly normal successive block sequence in
the canonical copy of $V$, contradicting Lemma~\ref{pre:nostrong}.
Therefore $q\notin\cP(z)$ for every $q>1$.  For $q\leq1$, density of
$G_q(z)$ would imply density of $G_1(z)$ and hence make $z$ itself strongly
normal, again impossible.  Consequently $\cP(z)=\varnothing$.

To apply Theorem~\ref{lor:weight} to the copy
$Y_I=Z_I\oplus\{0\}\subset V\subset E^*$, note that the canonical copy
of the $\ell_1$-saturated space $V$ contains no normalized strongly normal
sequence by Lemma~\ref{pre:nostrong}; hence neither does $Y_I$.
Moreover, weak-star null sequences in the canonical copy of $V$ are weakly
null in $V$, and therefore in $Y_I$.  Thus hypotheses
\textup{(H1)}--\textup{(H2)} of Theorem~\ref{lor:weight} hold, and the
Lorentz theorem gives $\cS(y)=I$.
The norm-null perturbation invariance of subsequence spectra,
Theorem~\ref{pert:thresholds}(ii), yields
\begin{equation}\label{joint:subsequence-spectrum}
 \cS(z)=I.
\end{equation}
Indeed, any good subsequence can be thinned until the coordinate
errors are $\ell_q$-summable for its specified exponent $q$;
this observation also gives a direct proof of the invocation.

To obtain the desired block exponents, set
\[
 t_k=\frac1{m_k}\sum_{i=1}^{m_k}e_{k,i}^X,
 \qquad
 u_k=\frac1{\delta_km_k}\sum_{i=1}^{m_k}z_{n(k,i)}.
\]
The sequence $(t_k)$ is isometrically equivalent to the coordinate
basis of $Z_J$.  The factor $1/\delta_k$ in the definition of $u_k$
is chosen precisely so that its second coordinate is $t_k$, of norm
one.  Its first coordinate has norm
$W_I(m_k)/(\delta_km_k)\leq2^{-k^2}<1$ by
\eqref{joint:block-lengths}.  Hence $\|u_k\|=1$, so $(u_k)$ is
already a normalized block sequence, and
\[
 \|u_k-(0,t_k)\|\leq2^{-k^2}.
\]
The roles of the parameters and blocks are summarized here:
\begin{center}
\small
\renewcommand{\arraystretch}{1.08}
\begin{tabular}{@{}p{.22\textwidth}p{.73\textwidth}@{}}
\toprule
Choice & Role in the construction\\
\midrule
$\delta_k=2^{-k}$ & Makes $\|z_{n(k,i)}-y_{n(k,i)}\|=\delta_k$ tend to zero,
so the subsequence spectrum is preserved.\\[3pt]
$m_k$ as in \eqref{joint:block-lengths} & Bounds the norm of the first coordinate of $u_k$
by $2^{-k^2}$; since $W_I(m_k)\geq1$, it also forces
$m_k\geq2^{k^2+k}$ and the obstruction \eqref{joint:empty-point-spectrum}.\\[3pt]
$n=n(k,i)$ & Enumerates finite rows successively, so their averages are
successive blocks of $z$.\\[3pt]
$t_k$ & The normalized row average in $X$, realizing the basis of $Z_J$
isometrically.\\[3pt]
$u_k$ & A normalized block of $z$ whose second coordinate is $t_k$ and
whose error from $(0,t_k)$ is at most $2^{-k^2}$.\\
\bottomrule
\end{tabular}
\end{center}

Every $f\in X^*$ restricts along the isometry $e_k^J\mapsto t_k$
to a functional on $Z_J$.  Hence for each $q\in J$ the sequence
$(0,t_k)$ has all-point $q$-summability on $E$.
Its difference from $(u_k)$ belongs to $\ell_q(V)$ for every
$q>0$.  We conclude that $G_q(u)=E$ for every $q\in J$, and
therefore $J\subset\cB(z)$.

It remains to exclude every $q\notin J$.  Suppose that a normalized
block sequence $b_j=(a_j,c_j)$ of $(z_n)$ were $q$-strongly normal.
If $q\leq1$, it would be strongly normal, contrary to the choice
of the canonical space $V$.
Let $q>1$.  The amplification lemma
(Proposition~\ref{pre:amplification}) and the absence of strongly
normal sequences imply $G_q(b)=E$.  Equivalently, for some $C$,
\begin{equation}\label{joint:upper-estimate}
 \Bigl\|\sum_j\lambda_jb_j\Bigr\|_V
 \leq C\|(\lambda_j)\|_{q'}
 \qquad((\lambda_j)\in c_{00}).
\end{equation}

If $\|a_j\|$ is bounded away from zero on a subsequence, pass to it
and choose $\eta>0$ with $\|a_j\|\geq\eta$. Applying the first-coordinate
projection in \eqref{joint:upper-estimate} with coefficients
$\lambda_j/\|a_j\|$ gives
\[
 \Bigl\|\sum_j\lambda_j\frac{a_j}{\|a_j\|}\Bigr\|_{Z_I}
 \leq C\Bigl\|\Bigl(\frac{\lambda_j}{\|a_j\|}\Bigr)\Bigr\|_{q'}
 \leq \frac{C}{\eta}\|(\lambda_j)\|_{q'}.
\]
For each $f\in Z_I^*$, duality between finite-dimensional
$\ell_{q'}$ and $\ell_q$ gives
\[
 \left(\sum_{j=1}^{N}
  \left|f\!\left(\frac{a_j}{\|a_j\|}\right)\right|^q\right)^{1/q}
 \leq\frac{C}{\eta}\|f\| \qquad(N\in\mathbb N).
\]
Letting $N\to\infty$ proves all-point $q$-summability on $Z_I^*$.
Apply the canonical case of Theorem~\ref{lor:weight}, with predual
$Z_I^*$ and embedding $J_{Z_I}:Z_I\to Z_I^{**}$; its hypotheses follow
from Lemma~\ref{pre:nostrong}. Since $q\notin J$ and $I\subseteq J$
imply $q\notin I$, this is a contradiction. No identification of the
whole space $V^*$ with $Z_I^*$ is involved. Consequently $\|a_j\|\to0$.
Since $\|b_j\|=\max(\|a_j\|,\|c_j\|)=1$, we have
$\|c_j\|=1$ eventually.

Each outer coordinate space of $X$ is finite dimensional, and
the supports of $(b_j)$ are successive finite coordinate intervals.
We can therefore pass to a subsequence for which the outer supports of
$(c_j)$ are disjoint: after one block has been chosen, the finitely many
rows it meets contain only finitely many scalar coordinates, so a sufficiently
late successive block starts beyond all those rows. Repeating this choice
gives disjoint outer supports.  Define its positive profiles
$d_j=\sum_k\|c_j^k\|_1 e_k^J$.  These are normalized disjoint
blocks in $Z_J$, and
\[
 \Bigl\|\sum_j\lambda_jd_j\Bigr\|_{Z_J}
 =\Bigl\|\sum_j\lambda_jc_j\Bigr\|_X
 \leq C\|(\lambda_j)\|_{q'}.
\]
This again contradicts the canonical case of Theorem~\ref{lor:weight},
now with predual $Z_J^*$, for the forbidden exponent $q\notin J$.
Thus $\cB(z)\subset J$.  Together with
\eqref{joint:subsequence-spectrum} and
\eqref{joint:empty-point-spectrum}, this proves all assertions
for proper $J$.

Finally let $J=(0,\infty)$.  Choose a normalized weakly null
$1$-unconditional basic sequence $a\subset E_a^*$ with
$\cS(a)=I$ and $\cB(a)=(0,\infty)$: for proper $I$ use
Corollary~\ref{pert:separation}, and for $I=(0,\infty)$ use an
orthonormal basis.  Choose a Lorentz sequence
$b\subset E_b^*$ with
$\cP(b)=\cS(b)=\cB(b)=\varnothing$.
In the dual of $E_a\oplus_1 E_b$ interleave them as
\[
 h_{2n-1}=(a_n,0),\qquad h_{2n}=(0,b_n).
\]
This is again normalized, weakly null and $1$-unconditional basic.
For each $q>0$,
\[
 G_q(h)=G_q(a)\oplus G_q(b),
\]
so $\cP(h)=\cP(a)\cap\cP(b)=\varnothing$.
Clearly $\cS(a)\subset\cS(h)$, by selecting odd terms.
A $q$-strongly normal subsequence of $h$ cannot contain infinitely
many even terms: those terms would themselves be $q$-strongly
normal, and projection onto $E_b$ would contradict
$\cS(b)=\varnothing$.  After deleting finitely many terms it
therefore consists of odd terms, so $q\in\cS(a)$.
Thus $\cS(h)=I$.  Every normalized block sequence of $a$ is a
normalized block sequence of $h$ using zero coefficients at even
positions.  Consequently $\cB(h)\supset\cB(a)=(0,\infty)$,
completing the proof.
\end{proof}

\begin{remark}\label{joint:scope}
Thus even the condition $\cP(y)=\varnothing$ imposes no further
restriction on the admissible pair $(\cS(y),\cB(y))$.
In the construction for proper $J$, all normalized block sequences
are excluded at the forbidden exponents, while the block sequence
$(u_k)$ explicitly realizes every allowed exponent.  Arbitrary
normalized blocks need not preserve the allowed exponents.
\end{remark}

\section{A universal selection model and complemented density defects}
\label{sec:jmsj}
Consider the Banach space
\[
 D=\{s=(s_n)\in\ell_\infty:\Delta s_n=s_n-s_{n-1}\longrightarrow0\},
 \qquad s_0=0,
\]
with the supremum norm. It is closed in $\ell_\infty$, since
$\Delta:\ell_\infty\to\ell_\infty$ is continuous and $c_0$ is closed.
The partial-sum map identifies $D$ isometrically with
\[
 E_0=\left\{x\in c_0:\sup_k\left|\sum_{n=1}^k x_n\right|<\infty\right\},
 \qquad \|x\|_{E_0}=\sup_k\left|\sum_{n=1}^k x_n\right|.
\]
The exponent-one bounded-run criterion for this model was proved in
\cite[pp.~392--394]{sliwa2012}. We retain that model and prove the
quasi-Banach, simultaneous-family, and closed-domain extensions below.
Define
\begin{equation}\label{jmsj:differences}
 y_1(s)=s_1,\qquad y_n(s)=\tfrac12\Delta s_n\quad(n\geq2).
\end{equation}
These are norm-one functionals and form a normal sequence. To see the norm
at $n\geq2$, use a finitely supported $s$ with $s_{n-1}=-1$, $s_n=1$ and
$\|s\|_\infty=1$. Normality follows from $\Delta s_n\to0$.

For an infinite $M\subset\mathbb N$, put
\[
 L(M)=\sup\{m:\{r,r+1,\ldots,r+m-1\}\subset M
                  \text{ for some }r\geq1\}.
\]
The selected evaluation sequence is always enumerated in increasing order.
Thus $G_A(y|M)$ denotes the set of $s\in D$ for which that sequence belongs
to $A$. Coordinate numbers in $M$ are not confused with their positions in
this enumeration.

\subsection{Finite-set estimates without convexity}
A symmetric quasi-Banach sequence ideal will mean a complete quasi-normed
linear space $A$ with
\begin{equation}\label{jmsj:qb-ideals}
 c_{00}\subset A\subsetneq c_0,
\end{equation}
whose quasi-norm is solid and invariant under all coordinate permutations.
Explicitly, $a\in A$ and $|b_n|\le |a_n|$ for every $n$ imply
$b\in A$ and $\|b\|_A\le\|a\|_A$.
The quasi-norm is absolutely homogeneous and vanishes only at zero. Completeness is
with respect to the quasi-norm topology. We use neither convexity nor the
Fatou property. Set
\[
 \varphi_A(k)=\left\|\sum_{i=1}^k e_i\right\|_A,
 \qquad H_A(N)=\sum_{k=1}^N\frac1{\varphi_A(k)}.
\]

\begin{lemma}\label{jmsj:fundamental}
For every such $A$, the function $\varphi_A$ is nondecreasing and tends to
infinity. For every $a\in A$ and every nonempty finite set $F\subset\mathbb N$,
\begin{equation}\label{jmsj:finite-set}
 \sum_{n\in F}|a_n|\leq H_A(|F|)\|a\|_A.
\end{equation}
Furthermore $H_A(N)/N\to0$.
\end{lemma}
\begin{proof}
Solidity proves monotonicity. Suppose $\varphi_A(k)\leq C$ for all $k$.
For every finitely supported $u$, solidity and permutation invariance give
$\|u\|_A\leq C\|u\|_\infty$. Hence the finite truncations of any $x\in c_0$
are Cauchy in $A$. They converge in $A$ by completeness. The estimate
\[
 |a_n|\varphi_A(1)\leq\|a\|_A
\]
shows that all coordinates of the limit equal $x_n$, so $x\in A$.
This would imply $A=c_0$, contrary to \eqref{jmsj:qb-ideals}.

Arrange the finitely many numbers $(|a_n|)_{n\in F}$ in decreasing order
$c_1\geq\cdots\geq c_N$, where $N=|F|$. On the positions of the $k$ largest
of these numbers, the constant $c_k$ is dominated by $|a|$. Therefore
$c_k\varphi_A(k)\leq\|a\|_A$. Summing these inequalities yields
\eqref{jmsj:finite-set}. This argument needs only finite masks; no global
rearrangement assumption beyond permutation invariance is used.
Finally, $1/\varphi_A(k)\to0$, and hence its arithmetic means tend to zero.
\end{proof}

\begin{remark}
For a Banach ideal the sharper estimate
\[
 \sum_{n\in F}|a_n|
 \leq\frac{|F|}{\varphi_A(|F|)}\|a\|_A
\]
follows by convex averaging of finite permutations. That sharper estimate
is not used for quasi-norms. The weaker bound
\eqref{jmsj:finite-set} is sufficient for every argument below.
\end{remark}

\subsection{The bounded-run case}
\begin{proposition}\label{jmsj:good}\label{jmsj:c0quotient}
Suppose $L=L(M)<\infty$. Then the subspace
\[
 D_{\mathrm{fin}}(M)=\{s\in D:y_n(s)=0
           \text{ for all but finitely many }n\in M\}
\]
is dense in $D$. In particular, for every nonempty family $\mathscr A$ of
sequence ideals containing $c_{00}$, with no restriction on its cardinality,
\[
 \overline{\bigcap_{A\in\mathscr A}G_A(y|M)}^{\,D}=D.
\]
Let $K_M=\bigcap_{n\in M}\ker y_n$. The operator
\[
 \Gamma_M:D\longrightarrow c_0(M),\qquad
 \Gamma_M(s)=(y_n(s))_{n\in M},
\]
has a bounded linear right inverse $R_M$ with $\|R_M\|\leq2L$.
Consequently $D/K_M$ is isomorphic to $c_0(M)$, and the whole selected
sequence is weak-star Schauder basic.
\end{proposition}
\begin{proof}
For $s\in D$, let $\eta_N=\sup_{n\geq N}|\Delta s_n|$, which tends to zero.
On every maximal run $\{r,\ldots,r+\ell-1\}\subset M$ with $r>N$, replace
all values of $s$ by $s_{r-1}$. At the next index, which is outside $M$,
retain the original value of $s$. The uniform change is at most $L\eta_N$.
Each new reconnecting increment is a sum of at most $L+1$ old increments,
so it tends to zero in the tail. The resulting sequence belongs to $D$ and
has only finitely many nonzero selected increments. These approximations
prove density and the common-domain assertion.

To construct $R_M$, take $a=(a_n)_{n\in M}\in c_0(M)$. Set $s_n=0$ outside
$M$. On a maximal run $\{r,\ldots,r+\ell-1\}$ define
\[
 s_n=\sum_{i=r}^n\gamma_i a_i\quad(r\leq n<r+\ell),
 \qquad \gamma_1=1,\quad \gamma_i=2\ (i\geq2).
\]
Then $\|s\|_\infty\leq2L\|a\|_\infty$ and $s\in c_0\subset D$.
At every selected coordinate $y_n(s)=a_n$. This defines the required linear
right inverse. The induced map $D/K_M\to c_0(M)$ is an isomorphism.
Moreover, $\Gamma_M^*:\ell_1(M)\to D^*$ is an isomorphic embedding onto
the weak-star closed annihilator $K_M^\perp$ and sends the coordinate
basis to $(y_n)_{n\in M}$. The weak-star topology on this range is the one
coming from $c_0(M)$, so the canonical basis expansions give the asserted
weak-star Schauder basicness.
\end{proof}

\subsection{Oscillation packages and a contractive left inverse}
The next lemma separates the geometric construction from all summability
estimates. It will also be used to distinguish different closed domains.

\begin{lemma}\label{jmsj:oscillations}
Suppose $L(M)=\infty$. Let $m_j\to\infty$ and $r_j\geq1$ be sequences of
positive integers. There are a linear isometry
$J:\ell_\infty\to D$, functionals $\tau_j\in D^*$ with $\|\tau_j\|\leq1$,
and a linear contraction $T:D\to\ell_\infty$ such that $TJ=I_{\ell_\infty}$.
They can be chosen with the following properties.
\begin{enumerate}[label=(\roman*)]
\item For each $j$, $\tau_j$ averages $r_j$ endpoint differences along
pairwise disjoint rises of length $m_j$, all on selected edges $n\in M$
with $n\geq2$.
\item If $\tau_j(s)\to0$, then $Ts=0$.
\item For every symmetric quasi-Banach ideal $A$ and $s\in G_A(y|M)$,
\begin{equation}\label{jmsj:tau-ideal}
 |\tau_j(s)|\leq\frac{H_A(m_jr_j)}{r_j}
                     \|(y_n(s))_{n\in M}\|_A.
\end{equation}
\item For $1\leq p<\infty$ and $s\in G_p(y|M)$,
\begin{equation}\label{jmsj:tau-p}
 |\tau_j(s)|\leq m_j^{1-1/p}r_j^{-1/p}
                     \|(y_n(s))_{n\in M}\|_p.
\end{equation}
\end{enumerate}
The projection $JT$ has norm one.
\end{lemma}
\begin{proof}
Arbitrarily long intervals contained in $M$ occur arbitrarily far out.
Indeed, deleting any fixed finite initial interval cannot bound all run
lengths if $L(M)=\infty$. Choose disjoint selected edge intervals
\[
 I_j=\{a_j+1,\ldots,a_j+2m_j(r_j+1)\}\subset M,
 \qquad a_j\geq1,
\]
with disjoint vertex intervals and a gap between successive ones.
On the vertices of the $j$th interval construct $h_j$ by starting at zero,
descending to $-1$ in $m_j$ edges, making $r_j$ complete rises from $-1$ to
$1$ and falls back to $-1$, each rise and fall taking $m_j$ edges, and
returning to zero in $m_j$ edges. Extend $h_j$ by zero. Thus
\[
 \|h_j\|_\infty=1,\qquad
 |\Delta h_j(n)|\leq2/m_j.
\]
The total number of edges is exactly $2m_j(r_j+1)$.

Choose a map $\kappa:\mathbb N\to\mathbb N$ with every fibre infinite;
for example, $\kappa(j)=1+\nu_2(j)$, where $\nu_2(j)$ is the largest
integer $v$ such that $2^v$ divides $j$. For $t\in\ell_\infty$ define
\[
 (Jt)_n=\sum_{j\geq1}t_{\kappa(j)}h_j(n).
\]
This is a pointwise definition on disjoint supports, not a claim of
norm convergence of a series. Its supremum norm is $\|t\|_\infty$.
On the $j$th interval its increments are at most
$2\|t\|_\infty/m_j$ in absolute value. They therefore tend to zero,
so $Jt\in D$ and $J$ is a linear isometry.

Let $u_{j,l}$ and $v_{j,l}=u_{j,l}+m_j$, $1\leq l\leq r_j$, be the
starting and ending vertices of the rises in the $j$th package. Put
\begin{equation}\label{jmsj:tau}
 \tau_j(s)=\frac1{2r_j}\sum_{l=1}^{r_j}
                         (s_{v_{j,l}}-s_{u_{j,l}}).
\end{equation}
Each endpoint difference is bounded by $2\|s\|_\infty$, so
$\|\tau_j\|\leq1$. Also $\tau_j(Jt)=t_{\kappa(j)}$.
Let $F_j$ be the union of the $m_jr_j$ rising edges. Since these edges
have indices at least two, \eqref{jmsj:differences} gives
\begin{equation}\label{jmsj:tau-edge}
 |\tau_j(s)|\leq\frac1{2r_j}\sum_{n\in F_j}|\Delta s_n|
               =\frac1{r_j}\sum_{n\in F_j}|y_n(s)|.
\end{equation}
Apply Lemma~\ref{jmsj:fundamental} to the positions of $F_j$ in the
increasing enumeration of $M$ to get \eqref{jmsj:tau-ideal}.
H\"older's inequality, including the direct $\ell_1$ estimate at $p=1$,
gives \eqref{jmsj:tau-p}.

It remains to construct one bounded left inverse. Let
$\Lambda:\ell_\infty\to\mathbb R$ be a norm-one linear extension of the
usual limit functional on the space $c$ of convergent sequences.
Such an extension exists by the Hahn--Banach theorem. For each $k$ enumerate
$\kappa^{-1}(k)$ increasingly as $j(k,1),j(k,2),\ldots$ and define
\[
 (Ts)_k=\Lambda\bigl((\tau_{j(k,l)}(s))_{l\geq1}\bigr).
\]
Then $\|Ts\|_\infty\leq\|s\|_\infty$ and $T(Jt)=t$.
If $\tau_j(s)\to0$, each sequence to which $\Lambda$ is applied is null,
so $Ts=0$. Finally $(JT)^2=JT$, and $JT$ is a nonzero contraction fixing
$J\ell_\infty$; hence it is a projection of norm one.
\end{proof}

\begin{remark}
The operator $T$ uses one fixed generalized limit; ordinary limits of
$\tau_{j(k,l)}(s)$ need not exist for general $s\in D$. Thus the right
inverse $J$ and the finite averages are explicit, while the bounded
projection uses Hahn--Banach. No canonical choice of projection is asserted.
\end{remark}

\subsection{One projection for a small family of ideals}
Write $f\leq^*g$ if $f(j)\leq g(j)$ for all sufficiently large $j$.
The bounding number $\mathfrak b$ is the least cardinality of an unbounded
family in $(\mathbb N^{\mathbb N},\leq^*)$. In particular, every family of
fewer than $\mathfrak b$ functions has an eventual upper bound. For a
countable family, such a bound is given directly by diagonal maxima.
The same domination principle occurs in \cite[Corollary~11 and its
proof]{survey2018}, but the projection constructed below is a different
application of it.

\begin{theorem}\label{jmsj:common-projection}
Let $\mathscr A$ be a nonempty family of symmetric quasi-Banach sequence
ideals satisfying \eqref{jmsj:qb-ideals}, and assume $|\mathscr A|<\mathfrak b$.
If $L(M)=\infty$, there are a linear isometry $J:\ell_\infty\to D$ and
a norm-one operator $T:D\to\ell_\infty$ such that
\begin{equation}\label{jmsj:splitting}
 TJ=I_{\ell_\infty},\qquad
 T(G_A(y|M))=\{0\}\quad(A\in\mathscr A).
\end{equation}
Let
\begin{equation}\label{jmsj:common-defect}
 F_{\mathscr A,M}=
 \overline{\operatorname{span}\bigcup_{A\in\mathscr A}G_A(y|M)}^{\,D}.
\end{equation}
Then $J\ell_\infty$ is a norm-one complemented subspace of $D$ with trivial intersection
with $F_{\mathscr A,M}$, and
\begin{equation}\label{jmsj:distance}
 \operatorname{dist}(Jt,F_{\mathscr A,M})=\|t\|_\infty
 \quad(t\in\ell_\infty).
\end{equation}
Moreover, $D/F_{\mathscr A,M}$ contains a norm-one complemented isometric
copy of $\ell_\infty$. In particular its density character is $\mathfrak c$.
\end{theorem}
\begin{proof}
Take $m_j=j$. For each $A\in\mathscr A$, Lemma~\ref{jmsj:fundamental}
provides a function $f_A:\mathbb N\to\mathbb N$ such that
\[
 r\geq f_A(j)\quad\Longrightarrow\quad
 \frac{H_A(jr)}r\leq2^{-j}.
\]
Indeed, for fixed $j$ the ratio is
$jH_A(jr)/(jr)$ and tends to zero as $r\to\infty$.
Choose $f$ eventually dominating every $f_A$, and choose increasing
positive integers $r_j\geq f(j)$. Apply Lemma~\ref{jmsj:oscillations} to
these parameters. For a fixed $A\in\mathscr A$ and $s\in G_A(y|M)$,
\eqref{jmsj:tau-ideal} gives, for all sufficiently large $j$,
\[
 |\tau_j(s)|\leq2^{-j}\|(y_n(s))_{n\in M}\|_A.
\]
Consequently $Ts=0$. The exceptional initial indices may depend on $A$;
the operators $J,T$ do not depend on the choice of $A$ or $s$.
This proves \eqref{jmsj:splitting}.

Continuity of $T$ gives $F_{\mathscr A,M}\subset\ker T$.
For $u\in F_{\mathscr A,M}$,
\[
 \|Jt-u\|_\infty\geq\|T(Jt-u)\|_\infty=\|t\|_\infty.
\]
The choice $u=0$ proves \eqref{jmsj:distance}. The projection $JT$ has
norm one by Lemma~\ref{jmsj:oscillations}.

Let $q:D\to D/F_{\mathscr A,M}$ be the quotient map.
The map $\overline J=qJ$ is an isometry by \eqref{jmsj:distance}, and
$T$ induces $\overline T:D/F_{\mathscr A,M}\to\ell_\infty$ by
$\overline T(qs)=Ts$. This map is well defined and contractive, since
\[
 \|Ts\|_\infty=\|T(s-u)\|_\infty\le\|s-u\|_\infty
 \quad(u\in F_{\mathscr A,M}),
\]
and taking the infimum over $u$ gives $\|\overline T(qs)\|\le\|qs\|$.
Moreover $\overline T\,\overline J=I$, so
$\overline J\,\overline T$ is a norm-one projection onto
$\overline J\ell_\infty$.
The quotient therefore has density at least that of $\ell_\infty$, namely
$\mathfrak c$. The reverse bound follows from $D\subset\mathbb R^{\mathbb N}$,
whose cardinality is $\mathfrak c$.
\end{proof}

\begin{corollary}\label{jmsj:ideal}
For an infinite $M\subset\mathbb N$ the following alternatives hold.
If $L(M)<\infty$, then the intersection of the domains for any family of
ideals containing $c_{00}$ is dense. If $L(M)=\infty$, then for every
nonempty family of proper symmetric quasi-Banach ideals of cardinality
less than $\mathfrak b$, even the closed linear span of their domains is
proper, and its quotient in $D$ has a norm-one complemented isometric copy
of $\ell_\infty$.
In particular, for a single such ideal,
\[
 \overline{G_A(y|M)}^{\,D}=D\quad\Longleftrightarrow\quad L(M)<\infty.
\]
\end{corollary}
\begin{proof}
Combine Proposition~\ref{jmsj:good} and
Theorem~\ref{jmsj:common-projection}.
\end{proof}

\begin{corollary}\label{jmsj:all-exponents}
If $L(M)=\infty$ and
\[
 F_{<\infty,M}=
 \overline{\bigcup_{p>0}G_p(y|M)}^{\,D},
\]
then $D/F_{<\infty,M}$ contains a norm-one complemented isometric copy
of $\ell_\infty$ and has density character $\mathfrak c$.
\end{corollary}
\begin{proof}
The family $(\ell_k)_{k\geq1}$ is countable, and
$\bigcup_{p>0}\ell_p=\bigcup_{k\geq1}\ell_k$.
The corresponding domains increase with $p$, so their union is linear.
Theorem~\ref{jmsj:common-projection} applies.
\end{proof}

\begin{remark}
For a countable family $\mathscr A=\{A_1,A_2,\ldots\}$ there is no need
to mention cardinal invariants in the proof: at step $j$, choose $r_j$
large enough that
\[
 \max_{1\leq k\leq j}\frac{H_{A_k}(jr_j)}{r_j}\leq2^{-j}.
\]
The conclusions are stronger than separate nondensity of every domain,
stronger than nondensity of their union, and stronger than a single ball
avoiding that union. They concern its entire closed linear span.
For $\|t\|_\infty=1$, \eqref{jmsj:distance} gives the exact distance $1$.
The bound $|\mathscr A|<\mathfrak b$ additionally permits uncountable
families of ideals of size below $\mathfrak b$, when such cardinalities
exist. It provides one common projection, not separate projections for
individual ideals. If $\mathfrak b=\aleph_1$, the permitted families
are precisely the at most countable ones. No sharpness assertion at the
cardinality $\mathfrak b$ is made.
\end{remark}

\subsection{Two-sided complemented jumps in the closed-domain scale}
\label{jmsj:closed-domains}
The common run-length criterion does not mean that all closed summability
domains coincide. We obtain a stronger assertion than pairwise strict
inclusion: the scale has complemented jumps from both sides. For an
infinite $M\subset\mathbb N$ and positive exponents, write
\[
 F_{p,M}=\overline{G_p(y|M)}^{\,D},\qquad
 F_{<q,M}=\overline{\bigcup_{0<p<q}G_p(y|M)}^{\,D},\qquad
 F_{q+,M}=\bigcap_{p>q}F_{p,M}.
\]
The union is a linear subspace, since the domains increase with $p$.
Moreover,
\[
 F_{<q,M}=\overline{\bigcup_{0<p<q}F_{p,M}}^{\,D},\qquad
 F_{<q,M}\subset F_{q,M}\subset F_{q+,M}.
\]
The notation $q+$ means intersection of the higher closed domains;
it does not assert that this intersection equals $F_{q,M}$.
For $M=\mathbb N$ we write $F_p$, $F_{<q}$ and $F_{q+}$.

\begin{theorem}\label{jmsj:strict-scale}
For every infinite $M\subset\mathbb N$ and $0<p\leq1$,
\begin{equation}\label{jmsj:low-exponents}
 F_{p,M}=F_{1,M}=\overline{D_{\mathrm{fin}}(M)}^{\,D}.
\end{equation}
For $M=\mathbb N$ this common space is $c$, the space of convergent sequences.
If $L(M)=\infty$, then for every $1<q<\infty$ there are a linear isometry
$J_q:\ell_\infty\to D$ and a contraction $T_q:D\to\ell_\infty$ such that
\begin{equation}\label{jmsj:endpoint-splitting}
 J_q(\ell_\infty)\subset G_q(y|M),\qquad
 T_q(F_{<q,M})=\{0\},\qquad T_qJ_q=I_{\ell_\infty}.
\end{equation}
Consequently $F_{q,M}/F_{<q,M}$ contains a norm-one complemented
isometric copy of $\ell_\infty$, and
\begin{equation}\label{jmsj:endpoint-distance}
 \operatorname{dist}(J_qt,F_{<q,M})=\|t\|_\infty
 \quad(t\in\ell_\infty).
\end{equation}
The same conclusions hold with $F_{<q,M}$ replaced by $F_{p,M}$ for any
$0<p<q$. In particular, for $1<p<q<\infty$,
\[
 F_{1,M}\subsetneq F_{p,M}\subsetneq F_{<q,M}
 \subsetneq F_{q,M}\subsetneq D.
\]
There is also a right jump. If $L(M)=\infty$, then for every $q\ge1$
there are a linear isometry $J_q^+:\ell_\infty\to D$ and a contraction
$T_q^+:D\to\ell_\infty$ such that
\begin{equation}\label{jmsj:right-splitting}
 J_q^+(\ell_\infty)\subset\bigcap_{p>q}G_p(y|M),\qquad
 T_q^+(F_{q,M})=\{0\},\qquad T_q^+J_q^+=I_{\ell_\infty}.
\end{equation}
Consequently $F_{q+,M}/F_{q,M}$ contains a norm-one complemented
isometric copy of $\ell_\infty$, and
\begin{equation}\label{jmsj:right-distance}
 \operatorname{dist}(J_q^+t,F_{q,M})=\|t\|_\infty
 \quad(t\in\ell_\infty).
\end{equation}
All the quotient spaces asserted to contain $\ell_\infty$, as well as
$D/F_{q,M}$ for finite $q>0$, have density character $\mathfrak c$.
\end{theorem}
\begin{proof}
We first prove \eqref{jmsj:low-exponents}. Take $s\in G_1(y|M)$ and
$N\geq1$. Define
\[
 d_n^{(N)}=\sum_{\substack{k\in M\\N<k\leq n}}\Delta s_k,
 \qquad s_n^{(N)}=s_n-d_n^{(N)}.
\]
The selected increments are absolutely summable, so the sequence
$d^{(N)}$ converges and
\[
 \|s-s^{(N)}\|_\infty
 \leq\sum_{\substack{k\in M\\k>N}}|\Delta s_k|
 \longrightarrow0.
\]
Furthermore $s^{(N)}\in D$, and all its selected increments past $N$
vanish. Thus $s^{(N)}\in D_{\mathrm{fin}}(M)$. For $0<p\leq1$,
\[
 D_{\mathrm{fin}}(M)\subset G_p(y|M)\subset G_1(y|M),
\]
which proves equality of the closures. When $M=\mathbb N$, absolute
summability of all increments implies convergence, while
$D_{\mathrm{fin}}(\mathbb N)$ is exactly the eventually constant sequences.
Their uniform closure is $c$.

Now assume $L(M)=\infty$ and fix $q>1$. In
Lemma~\ref{jmsj:oscillations} use
\begin{equation}\label{jmsj:endpoint-parameters}
 m_j=2^{j^2},\qquad
 r_j=\left\lceil 2^{-j}m_j^{q-1}\right\rceil.
\end{equation}
Let $J_q,T_q$ and $\tau_j$ be its operators and endpoint averages.
For every fixed $1\leq p<q$,
\begin{equation}\label{jmsj:strict-decay}
 m_j^{1-1/p}r_j^{-1/p}
 \leq 2^{j/p}m_j^{(p-q)/p}
 =2^{(j-(q-p)j^2)/p}\longrightarrow0.
\end{equation}
Hence $s\in G_p(y|M)$ implies $\tau_j(s)\to0$ by
\eqref{jmsj:tau-p}, and therefore $T_qs=0$.
For $0<p<1$ use $G_p(y|M)\subset G_1(y|M)$, since $1<q$.
The same $T_q$ therefore vanishes on every domain with exponent below $q$,
and continuity implies $T_q(F_{<q,M})=\{0\}$.
There is no countability restriction on the exponents here: the explicit
parameters \eqref{jmsj:endpoint-parameters} work for every fixed $p<q$.
Uniform convergence as $p$ approaches $q$ is neither asserted nor needed:
the same bounded operator vanishes on each domain, and hence on the
closure of their union.

For the positive assertion, compute the $q$-energy of one amplitude-one
package $h_j$. The first and last slopes have absolute value $1/m_j$ on
$m_j$ edges each, while its $r_j$ rises and $r_j$ falls have absolute
value $2/m_j$ on $m_j$ edges each. Consequently
\begin{equation}\label{jmsj:energy}
 \sum_n|\Delta h_j(n)|^q
 =\bigl(2+2^{q+1}r_j\bigr)m_j^{1-q}.
\end{equation}
Since $r_j\leq 2^{-j}m_j^{q-1}+1$,
\[
 \bigl(2+2^{q+1}r_j\bigr)m_j^{1-q}
 \leq (2+2^{q+1})2^{-(q-1)j^2}+2^{q+1-j}.
\]
This bound is summable over $j$. The increment supports of the packages
are disjoint and lie in $M\setminus\{1\}$. For every $t\in\ell_\infty$,
\[
 \sum_n|\Delta(J_qt)_n|^q
 \leq\|t\|_\infty^q
       \sum_{j\geq1}\bigl(2+2^{q+1}r_j\bigr)m_j^{1-q}
 <\infty.
\]
Thus $J_qt\in G_q(y|M)$ for all bounded $t$, not just for finitely
supported coefficient sequences. This proves
\eqref{jmsj:endpoint-splitting}.

For $u\in F_{<q,M}$, contractivity gives
\[
 \|J_qt-u\|_\infty
 \geq\|T_q(J_qt-u)\|_\infty=\|t\|_\infty.
\]
The opposite inequality follows by taking $u=0$, proving
\eqref{jmsj:endpoint-distance}. Let
$Q:F_{q,M}\to F_{q,M}/F_{<q,M}$ be the quotient map. Then
\[
 \widetilde J_q=QJ_q,\qquad
 \widetilde T_q(Qs)=T_qs
\]
are well defined, $\widetilde J_q$ is an isometry, and the quotient-norm
argument used in Theorem~\ref{jmsj:common-projection} gives
$\|\widetilde T_q\|\leq1$. Moreover,
$\widetilde T_q\widetilde J_q=I$. Their product
$\widetilde J_q\widetilde T_q$ is the required norm-one projection.
Since $F_{p,M}\subset F_{<q,M}$ for every $p<q$, exactly the same argument
applies to the quotient by $F_{p,M}$. Choosing an intermediate exponent
$r$ with $p<r<q$ shows that $F_{p,M}\subsetneq F_{<q,M}$ whenever
$1\leq p<q$, because $F_{p,M}\subsetneq F_{r,M}\subset F_{<q,M}$.

For the right jump, fix $q\ge1$ and apply Lemma~\ref{jmsj:oscillations}
with the different parameters
\begin{equation}\label{jmsj:right-parameters}
 m_j=2^{j^2},\qquad r_j=\left\lceil j m_j^{q-1}\right\rceil.
\end{equation}
Denote the resulting operators by $J_q^+,T_q^+$ and the averages by
$\tau_j^+$. At exponent $q$, estimate \eqref{jmsj:tau-p} gives
\[
 |\tau_j^+(s)|\le j^{-1/q}\|(y_n(s))_{n\in M}\|_q\longrightarrow0
 \quad(s\in G_q(y|M)).
\]
Thus $T_q^+$ vanishes on $G_q(y|M)$ and, by continuity, on $F_{q,M}$.
This includes $q=1$, when $r_j=j$.

For any fixed $p>q$, the same calculation as in \eqref{jmsj:energy}
at exponent $p$, and the bound $r_j\le j m_j^{q-1}+1$, yield
\begin{equation}\label{jmsj:right-energy}
 \sum_n|\Delta h_j(n)|^p
 \le (2+2^{p+1})2^{-(p-1)j^2}
       +2^{p+1}j\,2^{-(p-q)j^2}.
\end{equation}
Both series in $j$ converge, since $p>q\ge1$. The disjoint increment
supports therefore give $J_q^+t\in G_p(y|M)$ for every
$t\in\ell_\infty$. The parameters are independent of $p$, so the same
range lies in $\bigcap_{p>q}G_p(y|M)\subset F_{q+,M}$. Together with
$T_q^+J_q^+=I$ this proves \eqref{jmsj:right-splitting}.
For $u\in F_{q,M}$,
\[
 \|J_q^+t-u\|_\infty\ge\|T_q^+(J_q^+t-u)\|_\infty=\|t\|_\infty;
\]
taking $u=0$ proves \eqref{jmsj:right-distance}. Apply the same induced-map
construction, now to the quotient $F_{q+,M}/F_{q,M}$, to obtain the
norm-one projection.

Finally $F_{q,M}\ne D$ and the assertion about $D/F_{q,M}$, for every
finite $q>0$, follow from Theorem~\ref{jmsj:common-projection} for the
singleton family $\{\ell_q\}$. The density lower bounds follow from the
embedded copies of $\ell_\infty$, and the upper bounds from
$|D|\leq\mathfrak c$.
\end{proof}

\begin{corollary}\label{jmsj:full-scale}
For the full sequence $M=\mathbb N$, one has
\[
 F_p=c\quad(0<p\leq1),\qquad
 c\subsetneq F_p\subsetneq F_q\subsetneq D
 \quad(1<p<q<\infty).
\]
For every $q>1$, both quotients
\[
 F_q/F_{<q}\qquad\text{and}\qquad F_{q+}/F_q
\]
contain norm-one complemented isometric copies of $\ell_\infty$.
At exponent one, $F_{1+}/c$ also contains such a copy.
In particular,
\[
 F_{<q}\subsetneq F_q\subsetneq F_{q+}\quad(q>1),\qquad
 c=F_1\subsetneq F_{1+}.
\]
For $0<q<1$, both one-sided spaces equal $F_q=c$.
Thus the increasing scale is neither left nor right continuous at any
exponent above one, and is not right continuous at exponent one.
\end{corollary}
\begin{proof}
Apply Theorem~\ref{jmsj:strict-scale} with $M=\mathbb N$. For $0<q<1$,
all lower domains have closure $c$, and the intersection over higher
exponents contains $c$ and includes an exponent in $(q,1]$, whose closed
domain is $c$.
\end{proof}

\begin{remark}
The two-sided conclusions hold for every selection with unbounded runs.
The operators necessarily depend on the threshold: even allowing the
right inverses to vary, a single left-limit operator $T$ cannot work at
two thresholds $1<q<r$, because
$J_q\ell_\infty\subset G_q(y|M)\subset F_{<r,M}$ would give
$TJ_q=0$ instead of the identity. Likewise, a single right-limit operator
cannot work at two thresholds $1\le q<r$, because
$J_q^+\ell_\infty\subset G_r(y|M)\subset F_{r,M}$.
This does not conflict with Theorem~\ref{jmsj:common-projection}, whose
right inverse is not required to belong to any of the annihilated domains.
These results concern the geometry of the closed domains of a fixed
sequence, not the classification of basic spectrum triples.
\end{remark}

For a concrete comparison, let $M=\{1,3,5,\ldots\}$ and $A=\ell_2$.
Every selected run has length one. Keeping $s$ unchanged up to a cutoff
and replacing each sufficiently late $s_{2k-1}$ by $s_{2k-2}$ makes all
but finitely many selected increments zero. The uniform error tends to
zero because $\Delta s_n\to0$; the reconnecting increments also tend to
zero. Thus $G_2(y|M)$ is dense, and the same approximants work for every
ideal containing $c_{00}$.
By contrast, $M=\N$ has unbounded runs: $G_2(y)$ is not dense and
$D/F_2$ contains a $1$-complemented isometric copy of $\ell_\infty$ by
Theorem~\ref{jmsj:common-projection}. More generally the full sequence has
$\cP(y)=\varnothing$, whereas its odd subsequence is good for every
positive exponent. Hence $\cS(y)=(0,\infty)$.

\section{Sharp perturbations for symmetric linear sequence ideals}
\label{sec:ideal-perturb}
In this section $A$ is a solid, permutation-invariant linear subspace of
$\mathbb R^{\mathbb N}$ satisfying
\begin{equation}\label{ideal:algebraic-assumptions}
 c_{00}\subset A\subsetneq c_0.
\end{equation}
Solidity means that $a\in A$ and $|b_n|\leq|a_n|$ for all $n$ imply $b\in A$.
Permutation invariance is understood for genuine bijections of $\mathbb N$.
In particular, bounded coordinatewise multiplication preserves membership
in $A$. No norm, completeness or local convexity is assumed. For bounded
$y\subset E^*$ put $G_A(y)=\{x\in E:(y_n(x))\in A\}$.

\begin{lemma}\label{ideal:minorant}
If $\varepsilon=(\varepsilon_n)\geq0$ and $\varepsilon\notin A$, then there
is $a\in c_0\setminus A$ with $0\leq a_n\leq\varepsilon_n$ such that both
$\operatorname{supp}a$ and its complement are infinite.
\end{lemma}
\begin{proof}
Suppose first that $\varepsilon\in c_0$. Partition $\mathbb N=P\sqcup Q$
into infinite sets. At least one of the zero-padded restrictions
$\varepsilon\mathbf1_P$, $\varepsilon\mathbf1_Q$ is outside $A$, since their
sum is $\varepsilon$. Choose it as $a$. Its support is infinite because
$c_{00}\subset A$, and its complement contains the other infinite half.

If $\varepsilon\notin c_0$, choose $\delta>0$ and an infinite $K\subset\mathbb N$
such that $\varepsilon_n\geq\delta$ for $n\in K$. Take $b\in c_0\setminus A$,
replace it by $|b|$, and scale it to have $\|b\|_\infty\leq\delta$.
The replacement preserves nonmembership by solidity. Splitting $b$ over
two infinite sets as above produces $c\in c_0\setminus A$ whose support
$T$ and complement $T^c$ are infinite. Choose an infinite $L\subset K$
with $L^c$ infinite. Bijections $T\to L$ and $T^c\to L^c$ together define
a permutation $\pi$ of $\mathbb N$. Set $a_{\pi(n)}=c_n$. Then $a\notin A$,
$a\in c_0$, and $a_n=0$ outside $L$, while $0\leq a_n\leq\delta\leq
\varepsilon_n$ on $L$.
\end{proof}

\begin{theorem}\label{ideal:perturb}
Let $A$ satisfy \eqref{ideal:algebraic-assumptions} and let $\varepsilon\geq0$ be
arbitrary. The following are equivalent:
\begin{enumerate}[label=(\roman*)]
\item $\varepsilon\in A$.
\item For every Banach space $E$ and all bounded sequences $y,z\subset E^*$
with $\|y_n-z_n\|\leq\varepsilon_n$, one has $G_A(y)=G_A(z)$.
\item For every Banach space $E$ and all normalized normal sequences
$y,z\subset E^*$ with $\|y_n-z_n\|\leq\varepsilon_n$, the density of
$G_A(y)$ is equivalent to that of $G_A(z)$.
\end{enumerate}
If $\varepsilon\notin A$, the failure in (iii) can be witnessed by normalized
weakly null sequences in a separable Hilbert space: one domain is dense and
the closure of the other is a hyperplane.
\end{theorem}
\begin{proof}
If $\varepsilon\in A$ and $x\in E$, solidity gives
$((y_n-z_n)(x))\in A$ because
$|(y_n-z_n)(x)|\leq\varepsilon_n\|x\|$. Linearity then gives equality of
the domains. Thus (i) implies (ii), and (ii) implies (iii).

For the converse take the sequence $a$ supplied by Lemma~\ref{ideal:minorant}.
Enumerate its support and complement as $\{j_k:k\geq1\}$ and $\{i_k:k\geq1\}$.
They exhaust $\mathbb N$; no leftover coordinates have to be treated.
In $H=\mathbb R e_0\oplus_2\ell_2$, identified with its dual, put
\[
 y_{i_k}=y_{j_k}=e_k,\qquad
 z_{i_k}=e_k,\qquad
 z_{j_k}=\frac{e_k+a_{j_k}e_0}{\sqrt{1+a_{j_k}^2}}.
\]
The sequences are normalized and weakly null: each $e_k$ occurs at only two
indices, and $a_n\to0$. The error is zero at $i_k$, and at $j_k$ it satisfies
\begin{align*}
 \|z_{j_k}-y_{j_k}\|^2
 &=2-\frac{2}{\sqrt{1+a_{j_k}^2}}\\
 &=\frac{2a_{j_k}^2}
 {\sqrt{1+a_{j_k}^2}(\sqrt{1+a_{j_k}^2}+1)}
 \leq a_{j_k}^2\leq\varepsilon_{j_k}^2.
\end{align*}
For $y$, the dense subspace $\mathbb R\oplus c_{00}$ is contained in $G_A(y)$.

Now take $x=(t,b)\in G_A(z)$ and set $d=(z_n(x))\in A$.
Masking $d$ to the $i_k$ coordinates gives $u\in A$ with $u_{i_k}=b_k$
and $u_{j_k}=0$. The permutation swapping each pair $(i_k,j_k)$ gives
$v\in A$ with $v_{j_k}=b_k$ and $v_{i_k}=0$. Multiply $d$ by the bounded
sequence which equals $1$ at $i_k$ and $\sqrt{1+a_{j_k}^2}$ at $j_k$.
Subtracting $u+v$ leaves exactly $ta$. Hence $ta\in A$, and $a\notin A$
forces $t=0$. Conversely, $(0,b)\in G_A(z)$ for every finitely supported $b$.
We have proved
\[
 \overline{G_A(y)}^{\,H}=H,\qquad
 \overline{G_A(z)}^{\,H}=\{0\}\oplus\ell_2.
\]
This contradicts (iii) whenever $\varepsilon\notin A$.
\end{proof}

\begin{remark}
Only linearity, solidity, genuine coordinate permutations, and
$c_{00}\subset A\subsetneq c_0$ are used. No completeness or local convexity
is needed, and the error bounds need not tend to zero. Taking $A=\ell_p$
for any $p>0$ recovers the fixed-exponent part of
Theorem~\ref{pert:thresholds}. The earlier direct proof is retained as an explicit
paired-coordinate example illustrating the three-part sharpness theorem.
\end{remark}

\section{All admissible pairs in one fixed dual}
\label{sec:fixed}
We finish by observing that the joint classification may be realized in
one fixed dual.  K\k{a}kol \cite{kakol2026} recently constructed, for
$|\Gamma|\ge\mathfrak c$, normalized weakly null sequences in
$\ell_1(\Gamma)^*$ with no strongly normal subsequence.  The case $(\cS,\cB)=(\varnothing,\varnothing)$ of
Corollary~\ref{fixed:joint} therefore recovers, for
$\Gamma=\mathfrak c$, the existence phenomenon of \cite{kakol2026} by
an independent argument.  The result below strengthens the fixed-dual
picture in a different direction by realizing every admissible nested
pair, retaining $1$-unconditional basic structure and allowing
$\cP=\varnothing$ simultaneously.  Put $\mathfrak c=2^{\aleph_0}$ and
$U=\ell_1(\mathfrak c)$.

\begin{lemma}[A common metric quotient]\label{fixed:quotient}
Every Banach space $F$ with $|B_F|\le\mathfrak c$ is a metric quotient
of $U$.
\end{lemma}
\begin{proof}
Choose an injection $\iota:B_F\to\mathfrak c$ and define
$Qe_{\iota(b)}=b$ for $b\in B_F$, and $Qe_\gamma=0$ on the remaining
coordinates.  Absolute summability extends this to a contraction
$Q:U\to F$.  For $f\ne0$, the vector
$\|f\|e_{\iota(f/\|f\|)}$ maps to $f$ and has norm $\|f\|$, so $Q$ is a
metric quotient.
\end{proof}

\begin{lemma}[Metric quotient transfer]\label{fixed:transfer}
Let $Q:E\to F$ be a metric quotient and $y=(y_n)\subset F^*$ bounded.
Then, for every $q>0$,
\[
 G_q((Q^*y_n))=Q^{-1}(G_q(y)).
\]
Consequently density of these domains, the subsequence spectrum, and,
for basic sequences, the normalized block spectrum are preserved by
$Q^*$.  Weak nullity, basicness and $1$-unconditionality are preserved as
well.
\end{lemma}
\begin{proof}
The domain identity is immediate from $(Q^*y_n)(x)=y_n(Qx)$.
For $A\subset F$, if $A$ is dense, every nonempty open $V\subset E$
has nonempty open image $Q(V)$, which meets $A$; hence $Q^{-1}(A)$ is
dense. Conversely, if $Q^{-1}(A)$ is dense, the inverse image of each
nonempty open subset of $F$ meets $Q^{-1}(A)$, so $A$ is dense.
The adjoint $Q^*$ is an isometry. In particular, for every finite
coefficient family,
\[
 \Bigl\|\sum_n a_nQ^*y_n\Bigr\|_{E^*}
 =\Bigl\|\sum_n a_ny_n\Bigr\|_{F^*}.
\]
Thus the same coefficients produce normalized successive blocks on
one side exactly when they do on the other. Applying the domain
identity to these blocks proves both inclusions between the block
spectra; the subsequence assertion is immediate. Basicness and
$1$-unconditionality are preserved by the same norm identity.
Weak nullity is preserved by the bounded linear map $Q^*$; equivalently,
for $x^{**}\in E^{**}$ one has
$x^{**}(Q^*y_n)=(Q^{**}x^{**})(y_n)\to0$.
\end{proof}

\begin{corollary}[Joint selection spectra in a fixed Schur predual]
\label{fixed:joint}
Let
\[
 \mathfrak A=\{\varnothing,(0,\infty),(1,\infty)\}
 \cup\{[r,\infty),(r,\infty):r>1\}.
\]
The pairs $(\cS(y),\cB(y))$ arising from normalized Schauder basic
sequences in $U^*=\ell_\infty(\mathfrak c)$ are exactly
\[
 \{(I,J):I,J\in\mathfrak A,\ I\subseteq J\}.
\]
Every such pair has a normalized weakly null $1$-unconditional Schauder
basic realization; the realization may also be chosen with
$\cP(y)=\varnothing$.
\end{corollary}
\begin{proof}
Necessity is Theorem~\ref{joint:classification}.  For sufficiency use the
realization constructed there, say in $F^*$.  Every ambient Banach space \(F\) occurring in that construction has
\(|B_F|\le\mathfrak c\).  In the proper-\(J\) case one has
\(F=V^*\) with \(V\) separable, so \(|F|\le\mathfrak c\); the auxiliary
Hilbert, \(c_0\), Lorentz-dual, and finite direct-sum constructions in
the remaining case have the same cardinal bound.  Hence
Lemma~\ref{fixed:quotient} gives a metric
quotient $U\to F$, and Lemma~\ref{fixed:transfer} transfers the sequence
isometrically to $U^*$ while preserving all three density spectra and
its unconditional basic structure.
\end{proof}

\begin{remark}
The fixed predual $U=\ell_1(\mathfrak c)$ has the Schur property.
Thus the full selection classification already occurs in the dual of a
single familiar nonreflexive space whose weak sequential structure is
very rigid.  More generally, if $|\Gamma|\ge\mathfrak c$, the coordinate
projection $\ell_1(\Gamma)\twoheadrightarrow\ell_1(\mathfrak c)$ is a
metric quotient, so Lemma~\ref{fixed:transfer} transfers every realization
of Corollary~\ref{fixed:joint} to $\ell_1(\Gamma)^*$.

The cardinal $\mathfrak c$ is not claimed to be optimal.  By
Corollary~\ref{pre:small-density}, if $\kappa<\mathfrak b$, then every
normal sequence in $\ell_1(\kappa)^*$ has $\cS=(0,\infty)$, so no proper
subsequence spectrum can occur there.  Consequently any universal
fixed-dual realization of all admissible pairs must have
$\kappa\ge\mathfrak b$, while Corollary~\ref{fixed:joint} shows that
$\kappa=\mathfrak c$ suffices.  The intermediate range
$\mathfrak b\le\kappa<\mathfrak c$ remains an open cardinal problem;
compare \cite{kakol2026}.
\end{remark}

\section{Further questions}\label{sec:questions}
Theorem~\ref{orig:classification} classifies $\cP$ for arbitrary normal
sequences, but its Hilbert-space realizations are not basic. The basic
triple problem therefore remains distinct: determine all triples
$(\cP(y),\cS(y),\cB(y))$ for normalized normal Schauder basic sequences.
Theorem~\ref{joint:classification} realizes $(\varnothing,I,J)$ for every
admissible nested pair, Theorem~\ref{lor:classification} realizes
$(I,I,I)$, and Corollary~\ref{pert:separation} realizes
$(I,I,(0,\infty))$ for every proper admissible $I$.
Basic examples with $\varnothing\ne\cP(y)\subsetneq\cS(y)$ already
follow from the Lorentz models; their existence is not an open part of
the triple problem. Indeed, for normal sequences
$u\subset E_1^*$ and $v\subset E_2^*$ define, in
$(E_1\oplus_1E_2)^*$,
\[
 h_{2n-1}=(u_n,0),\qquad h_{2n}=(0,v_n).
\]
Then $G_p(h)=G_p(u)\times G_p(v)$, and hence
\[
 \cP(h)=\cP(u)\cap\cP(v),\qquad
 \cS(h)=\cS(u)\cup\cS(v).
\]
For the second identity, a good infinite subsequence has infinitely many
terms of at least one parity. Restricting to those terms preserves
density and gives a good subsequence in the corresponding factor;
conversely, either factor can be selected alone. If both sequences are
weakly null and $1$-unconditional basic, so is their normalized
interleaving, since its finite-sum norm is the maximum of the two
coordinate-sum norms. Taking the canonical models of
Corollary~\ref{lor:examples} with weights $n^{-1/2}$ and $n^{-1/3}$ gives
\[
 \varnothing\ne\cP(h)=(3,\infty)
 \subsetneq (2,\infty)=\cS(h).
\]
This answers the existence question, not the classification of all basic
triples; no general formula for the block spectrum of an interleaving is
asserted here. If $\cP(y)$ meets $(0,1]$, then
$\cS(y)=\cB(y)=(0,\infty)$. Without basicness,
Theorem~\ref{orig:classification} also realizes $[1,\infty)$ and
$(1/2,\infty)$ as $\cP$.

Theorem~\ref{lor:open-space} gives closed spaces of sequences with exact
open thresholds $(r,\infty)$, $r\ge1$. Analogous spaceability statements
for closed thresholds $[r,\infty)$ and for the empty spectrum remain to
be investigated. Another extension is to replace the $\ell_p$ scale by
an ordered scale of symmetric sequence ideals: which hypotheses on the
scale yield an analogue of the selection classification?

Section~\ref{sec:jmsj} settles the quasi-Banach extension of the run-length
criterion. It also gives a common complemented obstruction for every
family of fewer than $\mathfrak b$ ideals. We do not determine whether this
cardinal bound is optimal. It is also natural to seek analogous criteria
on locally finite graphs, where bounded runs would have to be replaced
by a suitable geometric condition. The complemented jumps of
Theorem~\ref{jmsj:strict-scale} concern one fixed normal sequence; they do
not settle the basic triple problem.

The least cardinal $\kappa$ for which all admissible pairs can be realized
in $\ell_1(\kappa)^*$ lies between $\mathfrak b$ and $\mathfrak c$, by
Corollaries~\ref{pre:small-density} and \ref{fixed:joint}.
When $\mathfrak b=\mathfrak c$, these bounds coincide. The intermediate
case $\mathfrak b<\mathfrak c$ is not resolved by the constructions here.
The results above are proved in ZFC. No assumption such as
$\mathfrak b=\mathfrak c$ is needed for the classification or the
fixed-dual realization; that equality only closes the stated gap
between the lower and upper cardinal bounds.

\section*{Declaration of generative AI and AI-assisted technologies in the manuscript preparation process}
During the preparation of this work, the author used OpenAI's ChatGPT for
exploratory mathematical discussion, including proposed arguments, manuscript
organization, language refinement, consistency checks, and exposition.
AI-generated suggestions and review reports are not independent mathematical
verification. The author is responsible for the mathematical content,
references, and final text.

\enlargethispage{4\baselineskip}

\end{document}